\documentclass[twoside,11pt]{article}

\usepackage[abbrvbib,preprint]{jmlr2e}

\usepackage[T1]{fontenc}
\usepackage[utf8]{inputenc}
\usepackage{amsmath}
\usepackage{ulem}
\usepackage{mathtools}
\usepackage{comment}
\usepackage{multicol}
\usepackage{bbm}
\usepackage{tabularx}
\usepackage{hyperref}
\usepackage{enumitem}
\usepackage{colortbl}
\usepackage[usenames,dvipsnames]{xcolor}
\usepackage{tikz}
\usepackage{tikz-cd}
\usetikzlibrary{trees,automata,positioning,calc}
\usepackage{stmaryrd}
\usepackage[framemethod=TikZ]{mdframed}
\usepackage{subcaption}
\usepackage{array}
\usepackage{multirow}
\usepackage{cleveref}

\crefname{proposition}{proposition}{propositions}
\Crefname{proposition}{Proposition}{Propositions}

\crefname{lemma}{lemma}{lemmas} 
\Crefname{lemma}{Lemma}{Lemmas}

\crefname{corollary}{corollary}{corollaries} 
\Crefname{corollary}{Corollary}{Corollaries}

\crefname{remark}{remark}{remarks} 
\Crefname{remark}{Remark}{Remarks}

\crefname{definition}{definition}{definitions} 
\Crefname{definition}{Definition}{Definitions}

\crefname{conjecture}{conjecture}{conjectures} 
\Crefname{conjecture}{Conjecture}{Conjectures}

\crefname{axiom}{axiom}{axioms} 
\Crefname{axiom}{Axiom}{Axioms}

\newcommand{\sA}{\mathcal{A}}

\newcommand{\sD}{\mathcal{D}}

\newcommand{\sF}{\mathcal{F}}

\newcommand{\p}{\bullet}
\newcommand{\pp}{ { \p\p } }

\newcommand{\width}{\ensuremath{W}}

\newcommand{\card}[1]{\ensuremath{\# #1}}
\newcommand{\leafs}[1]{\ensuremath{\textbf{L}(#1)}}

\newcommand{\rootcell}{\ensuremath{\mathfrak{t}} }

\newcommand{\eins}{\ensuremath{\mathbf{V}}}
\newcommand{\zwei}{\ensuremath{\mathbf{S}}}
\newcommand{\einsarg}[1]{\ensuremath{ \eins}}

\newcommand{\zweiarg}[1]{\ensuremath{ \zwei }}
\newcommand{\zweiargneu}[1]{\ensuremath{ \zwei^{(#1)}}}

\newcommand{\rsrf}{\text{RSRF}}
\newcommand{\Tgeneral}{\widehat{T}}
\newcommand{\bdiff}{\ensuremath{\rho}}

\newcommand{\N}{\ensuremath{\mathbb{N}}}   
\newcommand{\R}{\ensuremath{\mathbb{R}}}   
\newcommand{\E}{\ensuremath{\mathbf{E}}}   
\newcommand{\PP}{\ensuremath{\mathbf{P}}}   
\newcommand{\1}{\ensuremath{\mathbbm{1}}}   

\newcommand{\dx}{\ensuremath{\mathrm{d}}}
\newcommand{\var}{\ensuremath{\mathbf{Var}}}

\newcommand{\quot}[1]{\text{``#1''}}
\newcommand{\bt}{\ensuremath{\mathbf{t}}}

\newcommand{\bc}{\ensuremath{\mathbf{c}}}
\newcommand{\bq}{\ensuremath{\mathbf{q}}}

\newcommand{\bR}{\ensuremath{\mathbf{R}}}
 
\newcommand{\iid}{\ensuremath{\overset{\text{iid}}{\sim}}}
\makeatletter
\newcommand*{\rom}[1]{\expandafter\@slowromancap\romannumeral #1@}

\DeclarePairedDelimiterX{\inner}[2]{\langle}{\rangle}{#1, #2} 

\makeatother

\newtheorem{manualtheoreminner}{Condition}
\newenvironment{condition}[1]{%
  \renewcommand\themanualtheoreminner{#1}%
  \manualtheoreminner
}{\endmanualtheoreminner}

   \newenvironment{proofname}[1]{\par\noindent{\bf {#1} }}{\hfill\BlackBox\\[2mm]}

\makeatletter                              
\def\blfootnote{\xdef\@thefnmark{}\@footnotetext}
\makeatother
\setenumerate{label={\normalfont(\alph*)}} 

\allowdisplaybreaks

\usepackage{lastpage}

\jmlrheading{??}{2024}{1-\pageref{LastPage}}{2/24}{??}{21-0000}{Ricardo Blum, Munir Hiabu, Enno Mammen, Joseph Theo Meyer}

\ShortHeadings{Consistency of Random Forest Type Algorithms}{Blum, Hiabu, Mammen and Meyer}
\firstpageno{1}

\begin{document}

\title{Consistency of Random Forest Type Algorithms under a Probabilistic Impurity Decrease Condition}

\author{\name Ricardo Blum \email ricardo.blum@uni-heidelberg.de \\
       \addr Institute for Mathematics\\       Heidelberg University\\
       Im Neuenheimer Feld 205\\
       69120 Heidelberg, Germany
       \AND
       \name Munir Hiabu \email mh@math.ku.dk \\
       \addr Department of Mathematical Sciences\\ University of Copenhagen\\Universitetsparken 5\\
       2100 Copenhagen Ø, Denmark
       \AND
       \name Enno Mammen \email mammen@math.uni-heidelberg.de \\
        \addr Institute for Mathematics\\
       Heidelberg University\\
       Im Neuenheimer Feld 205\\
       69120 Heidelberg, Germany
       \AND 
       \name Joseph Theo Meyer \email Joseph-Theo.Meyer@uni-heidelberg.de \\
       \addr Institute for Mathematics\\
       Heidelberg University\\
       Im Neuenheimer Feld 205\\
       69120 Heidelberg, Germany
       }

\editor{My editor}

\maketitle

\begin{abstract}
This paper derives a unifying theorem establishing consistency results for a broad class of tree-based algorithms. It improves current results in two aspects. First of all, it can be applied to algorithms that vary from traditional Random Forests due to additional randomness for choosing splits, extending
split options, allowing partitions into more than two cells in a single iteration
step, and combinations of those.
In particular, we prove consistency for  Extremely Randomized Trees,  Interaction Forests and Oblique Regression Trees using our general theorem.
Secondly, it can be used to demonstrate consistency for a larger function class compared to previous results on Random Forests if one allows for additional random splits. Our results are based on the extension of the recently introduced notion of sufficient impurity decrease to a probabilistic sufficient impurity decrease condition. 
\end{abstract}


\newpage

    \section{Introduction}\label{sec:introduction}
    \noindent Decision tree ensembles have drawn remarkable attention within the field of Machine Learning over the last decades. In particular, Breiman's Random Forests \citeyearpar{breiman} has become popular among practitioners in many fields due to its intuitive description and predictive performance. Applications include finance, genetics, medical image analysis, to mention a few, see for example \citet{gu2020_finance, GDA2, Qi_bioinformatics, criminisi_mda, criminisi_unified}. In addition, various related algorithms have been developed and popularized. Their purpose usually is increased computational efficiency and/or improvement in accuracy in some situations. Examples are the Extremely Randomized trees algorithm by \citet{extra} that heavily relies on randomization of the usual CART criterion, Oblique Regression Trees \cite[see][Section 5.2]{breiman} that split cells along hyperplanes which are not necessarily perpendicular to one of the coordinate axes and Interaction Forests of \citet{interactionforests} which differs from the usual CART by allowing partitions through certain rectangular cuts along two directions (cf. \Cref{fig:intf}).\\
    While many modifications exist and are frequently used, most theoretical results developed in the literature focus on the traditional Random Forest algorithm or on simplified versions of the modifications. Our main goal in this paper is to formulate a consistency theorem which includes all algorithms mentioned above. We build upon methods developed by \citet{Chi} who showed that regression trees built with CART are consistent imposing a so-called SID (sufficient impurity decrease) condition on the regression function $m$. The SID condition requires that for every rectangular set, there exists a single rectangular split that sufficiently improves approximation. 
    The tree modifications mentioned above are based on alternative methods for obtaining a partition in an iteration step while keeping the tree structure. They include additional randomness for choosing splits, extending split options, allowing partitions into more than two cells in a single iteration step, and combinations of those. This brings new challenges. For example, since the possible splits are chosen at random for Extremely Random Trees given a partition, we will not necessarily have a split which improves approximation significantly in every step. 
    To face such obstacles, we will argue that the approach of \citet{Chi} can be extended to show consistency of algorithms for which impurity decrease only holds in a probabilistic sense. This probabilistic SID condition depends on the class of sets which may appear in a partition using the algorithm in question. The contributions of our paper can be summarized as follows.
    \begin{itemize}
        \item We provide a general framework for the theoretical analysis of various tree based algorithms by introducing a new probabilistic impurity decrease condition. We prove that under certain conditions on the splitting procedure, tree based methods are consistent under this framework.
        \item We demonstrate the significance of our theorem by proving consistency for Extremely Randomized Trees, Interaction Forests and Oblique Regression Trees using our general framework. To the best of our knowledge we are the first to prove consistency results for Extremely Randomized Trees and Interaction Forests.  
        \item We use our general framework to show consistency for a larger
function class compared to previous results on Random Forests derived by \citet{Chi} if one allows for additional random splits. 
    \end{itemize}
    In order to prove results for a larger class of possible regression functions than covered by the theory of \citet{Chi}, we introduce a new algorithm coined Random Split Random Forests (RSRF). The advantage of RSRF is that the algorithm produces partitions into rectangles as does Random Forests, while theoretically being able to find certain pure interactions which do not satisfy SID.
    Although our approach is quite general there exist tree-based methods in the literature which do not fit into our framework. These include for example XGBoost \citep[][]{xgboost}, Bayesian Additive Regression Trees \citep[][]{bart}, Random Planted Forests \citep[][]{rpf} and Iterative Random Forests \citep[][]{basu2018iterative}. \\
  Let us summarize existing work on consistency of regression trees and Random Forests that are close to the original Random Forest algorithm in the sense that the tree building process includes the CART optimization. \citet{sco-consistency} showed consistency in the case of additive regression functions with fixed dimension and continuous regression functions. \citet{wager2015adaptive} proved that certain CART-like trees are consistent in a high-dimensional regime. Recently, \citet{klu-large} showed that regression trees and forests are consistent for additive targets in a high-dimensional setup with the dimension growing sub-exponentially, under general conditions on the component functions. 
    Moving beyond additive models, \citet{Chi} have proven consistency at polynomial rate under the SID condition. Therein, the feature space dimension is allowed to grow polynomially with the sample size. A similar result was derived by \citet{syr} for the special case of binary response variables. \citet[Section 6]{klu-large} also proved consistency for CART for models with interactions, under an empirical impurity gain condition.
    For results on consistency of oblique trees, see \citet{oblique_consistency_klusowski} and \citet{oblique_consistency_zhan}.
      We emphasize that the results cited so far and our theoretical results are applicable for algorithms with double use of the data. The latter refers to using the same data when building a partition and providing estimates given a partition. This is one of the main challenges for theoretical discussions of Random Forests. There also exist several results in the literature for simplified or stylized variants of Random Forests that avoid the double use of data. The results cover Centered Random Forests \citep{biau, klu-sharp}, Purely Random Forests \citep{genuer2012variance,biau_devroye_lugosi} and Median Trees \citep{median-trees}. Furthermore, Mondrian Trees, as introduced by \citet{mondrian-algorithm}, have proven to be rate-optimal \citep{mondrian-consistency}. In the work of \citet{wager2018estimation}, double use of data is circumvented through the honesty condition and asymptotic normality for Honest Random Forests is derived. Distributional results for ensemble methods were also derived by \cite{mentch-hooker} and applied to construct confidence intervals and hypothesis tests for Random Forests. For a theoretical contribution on Iterative Random Forests, see \citet{behr} who studied consistency of interaction recovery.
\subsection{Organisation of the Paper}
    The paper is structured as follows. In \Cref{sec:probabilistic-sid-conditions} we introduce the general nature of our probabilistic SID condition and present the main consistency result. Here, we also discuss the applications to Extremely Randomized Trees, Oblique Trees and Interaction Forests. In \Cref{sec:consistency} we introduce the algorithm RSRF in order to demonstrate how the regression function class covered can be extended. There is a short outlook in \Cref{sec:outlook}. Proofs are deferred to the appendix.
\section{Consistency under Probabilistic SID Conditions}\label{sec:probabilistic-sid-conditions}

    Consider a nonparametric regression model
    \begin{align*}
		Y_i = m(X_i) + \varepsilon_i,
    \end{align*}
    $i =1,\dots, n,$ with i.i.d. data where $m:[0,1]^d \to \R$ is measurable and $(\varepsilon_i)_{i=1,\dots, n}$ is zero mean. Here, for simplicity, we assume that the support of $X_i$ is $[0,1]^d$. We are interested in estimating the function $m$. After introducing the definition of a general tree estimator in \Cref{subsec:general-tree-estimator}, we will present a probabilistic version of SID in \Cref{subsec:probabilistic-sid-condition}. \Cref{subsection:tree-conditions} includes conditions on the tree-based algorithm and \Cref{subsec:main-consistency-result} provides our main consistency result \Cref{thm:main3}.

\subsection{General Tree Estimator}\label{subsec:general-tree-estimator}
 The notion of impurity decrease is essential to greedy tree algorithms. 
\begin{definition}[Impurity decrease]\label[definition]{def:impurity_decrease}
Given a cell $t \subseteq[0,1]^d$ and a partition into $L$ sets, $P = \lbrace t_1, \dots, t_L \rbrace$ of $t$, we define
    \begin{align*}
            \zwei \left(t; P \right)
            = \sum_{l=1}^L \PP(X \in t_l |X\in t) \big[  \mu(t_l )- \mu(t) \big]^2.
       \end{align*}
    where $\mu(A) := \E[m(X)| X \in A]$. The quantity $\zwei$ is called \emph{impurity decrease}. Furthermore, define $\eins \left( t ;P \right) := \var(m(X)|X\in t) - \zwei \left(t;P \right)$. 
\end{definition}
    Let $P = \lbrace t_1, \dots, t_L \rbrace$ be a partition of $t$ and $\sF_L$ be the class of functions which are piecewise constant on the sets in $P$. We note that the quantity $\einsarg{L}$ is the conditional mean squared approximation error
    \begin{align*}
        \min_{f \in \sF_L} \E\left[ \left(m(X)-f(X) \right)^2|X\in t \right] = \einsarg{L}\left(t;P \right).
    \end{align*}
    Therefore, $\zweiarg{L}$ measures the improvement in quality of approximation for $m$, when switching from the class of constant functions to the class $\sF_L$. We also need the empirical version of $\zwei$.
\begin{definition}[Empirical impurity decrease]\label[definition]{def:impurity_decrease_empirical}
Given a cell $t \subseteq[0,1]^d$, a partition $P = \lbrace t_1, \dots, t_L \rbrace$ of $t$ and $x_i \in [0,1]^d,y_i \in \R$, $i=1,\dots,n$, we define
 \begin{align*}
            \widehat\zwei \left(t; P \right) = \sum_{ l = 1}^L \frac{ \card{t_{l}}}{\card{t} } \big[  \hat{\mu}(t_{l}) - \hat{\mu}(t) \big]^2,
       \end{align*}
       where $\hat{\mu}(t) = (\card{t})^{-1}\sum_{i : x_i \in t}y_i$, $\card{t}  = \# \{i : x_i \in t\}$. We use the convention $\frac{0}{0}=0.$
\end{definition}
Next, we introduce general tree estimators. In our setting, a tree is a partition of $[0,1]^d$. The sets in the partition are also called leaf nodes. We only consider iterative tree constructions where the updated tree depends on the leafs of the previous tree, but not on how the previous tree was built up. Throughout the paper, given a set $t \subseteq [0,1]^d$, we denote by $\mathcal{D}(t)$ a collection of partitions of $t$. If the collection of partitions depends on the realization of some random variables $\bR$, we call it a randomized collection. Randomization is assumed to be independent across $t$. Furthermore, we assume that randomization is independent of data. We consider general tree estimators $\hat m_{\Tgeneral}$ which depend on the data $(X_1,Y_1,\dots,X_n,Y_n)$ and possibly on the additional randomization. More precisely, we are interested in estimators of the form
        \begin{align}\label{eq:general_tree_estimator} \hat{m}_{ \Tgeneral } (x) :=   \sum_{ t \in  \Tgeneral} \Big(  \frac{1}{\card{ t }} \sum_{ {i : x_i \in t} } y_i \Big) \1_{x\in t}, \ x \in [0,1]^d, \end{align}
        where $\Tgeneral$ is a partition of $[0,1]^d$ that may depend on both $(X_1,Y_1,\dots,X_n,Y_n)$ and  $\bR$. We assume that the estimator is constructed subject to the following definition. Throughout the paper, it is notationally convenient to denote the root cell by $\rootcell := [0,1]^d$. 
        \begin{definition}[General tree estimator]\label[definition]{def:general_tree_algorithm}
            The tree $\Tgeneral$ is grown using the following algorithm. Start with $T_0=\{\rootcell\}$. For depth $m=1,2,\dots,k$ apply the following steps to all current leaf nodes $t \in T_{m-1}$.
              \begin{enumerate}
            \item Determine a randomized collection $\mathcal{D}(t)$ of partitions of $t$ into Borel sets of positive Lebesgue measure (by independently drawing a realization of the randomization variables).
            \item Choose the best partition $\mathcal{P}^{\text{best}}$, where
            \begin{align*}
                \mathcal{P}^{\text{best}} \in \underset{P \in \mathcal{D}(t)}{\arg\max}\ \widehat\zwei\left(t; P \right). 
            \end{align*}
            \item Add the sets in $\mathcal{P}^{\text{best}}$ to $T_{m}$.
        \end{enumerate}
        Set $\hat{T}:=T_{k}$.
        \end{definition}
         \Cref{def:general_tree_algorithm} states that in each step of the tree building, each current leaf cell is partitioned into some possibly random candidate partitions among which the best is chosen using the score $\widehat\zwei$. It is not specified how the partitions are constructed. Note that there is a $\arg\max$ above because $\widehat{\zwei}(t;P)$ can only take finitely many values. In case of ties, choose $\mathcal{P}^{\text{best}}$ randomly.
\begin{example}[CART-split]\label{def:sample-cart}  
    Let $t \subseteq [0,1]^d$ and 
    let $\mathcal{D}(t)$ denote the partitions obtained through a rectangular cut, i.e. $\mathcal{D}(t)$ is a the set  with elements $\{t_1,t_2\}$, where $t_1 = t_1(j,c),t_2 = t_2(j,c)$,
    \[ t_1(j,c):= \{ x\in t : x_{{j} } \leq {s} \}, \quad t_2(j,c) := \{x \in t : x_{{j}} > {s}\}, \]
    with $j\in \{1,\dots,d\}$ and $s \in \{x_j: x \in t\}$. Then, the algorithm greedily partitions $[0,1]^d$ using $\widehat{\zwei}$ as a criterion to choose the splits. This is known as the CART algorithm for regression \citep[see][]{cartbook}.    
\end{example}  

\subsection{Probabilistic SID Condition}\label{subsec:probabilistic-sid-condition}
The condition below is central to our theoretical analysis of tree-based algorithms.
\begin{condition}{(C1*)}[Probabilistic SID condition]\label{cond:sid-general}
    Let $L\geq 2$. Let $\mathcal{C}$ be a class of subsets of $[0,1]^d$ and, for any $t \in \mathcal{C}$, let $\mathcal{D}(t)\subseteq 2^\mathcal{C}$ be a randomized collection of partitions of $t$ into $L$ subsets. We say that a function $m$ satisfies a probabilistic SID condition if there exists $\delta \in (0,1]$ and $\alpha_1 \geq 1$ such that for all $t \in \mathcal{C}$ it holds that
    \begin{align*}
        \var(m(X)|X\in t) \leq \alpha_1 \sup_{ P \in \mathcal{D}(t)} \zwei(t;P) \ \text{with probability at least $\delta$.}
    \end{align*}
\end{condition}
Note that besides $\delta$ and $\alpha_1$, the condition depends on $L$ (via $\zwei$) and a randomized collection $\mathcal{D}(t)\subseteq 2^\mathcal{C}$ given a $t\in \mathcal{C}$. These terms come from the specific general forest-based algorithm considered, see Definition \ref{def:general_tree_algorithm}. It aims to ensure that there is sufficient improvement in approximation in a single step of the algorithm. \Cref{cond:sid-general} generalizes the SID (sufficient impurity decrease) condition of \citet{Chi} in three ways: First, our condition accounts for partitions of arbitrary size $L$ (recall that $L$ is in the definition of $\zwei$). Second, it adapts to general tree-based algorithms that are constructed by considering partitions $\mathcal D(t)$. Third, the inequality above is only assumed to hold in a probabilistic sense, resembling the fact that the partitions $\mathcal D(t)$ under consideration may be random. We show later that \Cref{cond:sid-general} can allow for strictly larger function classes compared to the SID condition of \citet{Chi} described in the following example.
\begin{example}[Example \ref{def:sample-cart} continued]\label{ex:CART_2}
    For conventional CART, choose $\mathcal{C}$ to be the rectangular subsets of $[0,1]^d$ and let $\mathcal{D}(t)$ be the partitions obtained through a single rectangular cut of $t$ as in \Cref{def:sample-cart}. Then for $\delta = 1$ ($L=2$) the condition becomes the non-probabilistic SID condition of \citet{Chi}. In this case, \citet{Chi} proved the following result.\\[2mm]
    \emph{Suppose \Cref{cond:sid-general} with $L=2$ and $\delta = 1$ as well as Conditions \ref{cond:C2}, \ref{cond:C3}, \ref{cond:C4} introduced below hold. 
    Then
    \begin{align*}
         \E\Big[ \big(m(X) -  \hat{m}_{\hat{T}} (X) \big)^2  \Big] \rightarrow 0
    \end{align*}
    holds, where $X$ is distributed as $X_1$ and independent of $D_n =(X_1,Y_1,\dots,X_n,Y_n)$.}\\[2mm]
    Our main contribution is a generalisation of this result to include probabilistic SID \ref{cond:sid-general}, general tree-based algorithms and $L>2$.
\end{example}

        \subsection{Conditions on the Algorithm}\label{subsection:tree-conditions}
        We impose certain geometric complexity conditions on the algorithm, relying on the grid defined below. While the conditions look very technical, they are easy to verify for most common algorithms which can be seen in the next section.

        \begin{definition}[Grid, see also \citeauthor{Chi}, \citeyear{Chi}]\label[definition]{def:grid}
            Let $g_n:= \lceil n^{1+\epsilon} \rceil$ with some $\epsilon>0$. Let $b_q := q/g_n$, $q = 0,\dots, g_n$, which form $g_{n}+1$ grid points along $[0,1]$. The intervals between grid points are denoted $I_q:= [b_q,b_{q+1})$ for $q=0,\dots,g_{n-2}$ and $I_{g_{n}-1} = [b_{g_n-1},1]$. Next, we define stripes  $G(j,q) := [0,1]^{j-1}\times I_q \times [0,1]^{d-j}$. Furthermore, $[0,1]^d$ is partitioned into boxes
            \begin{equation*}
                B_{\bq} := \bigtimes_{j=1}^d I_{q_j}, \text{ where } \bq = (q_1,\dots,q_d) \in Q^d,
            \end{equation*}
            where $Q:= \{0,\dots, g_{n}-1 \}$. We set $G_n := \big\lbrace B_{\bq} : \bq \in Q^d \big\rbrace$.
            Let $G_n^{\#}$ be the set of midpoints of all elements in $G_n$.
        \end{definition}
           In what follows, denote by $\mathcal{C}$ a class of subsets of $[0,1]^d$ that includes any set which can be obtained from any possible realization of the algorithm $\Tgeneral$, and assume that $\mathcal{C}_1(t)$ contains all subsets of $t$ obtainable from a single iteration of the algorithm, meaning $\mathcal{D}(t)\subseteq 2^{\mathcal{C}_1(t)}$ for all realizations of $\mathcal{D}(t)$. Note that the conditions below indirectly effect the randomized partitions $\mathcal{D}(t)$ from Definition \ref{def:general_tree_algorithm} since they restrict $\mathcal{C}_1(t)$.
        \begin{condition}{(Tree-1)}\label{cond:T_dim}
         There exists some $B,\beta> 0$ such that
            \begin{align}\label{cardinality-bound}
                \# \{ G^{\#}_n \cap t' : t' \in \mathcal{C}_1(t) \} \leq Bn^{\beta}, \text{ uniformly over $t\in \mathcal{C}$}.
            \end{align}
        \end{condition}
        \Cref{cond:T_dim} is a complexity bound for the number of possible sets that can be constructed through a single iteration of the algorithm. Next, we define the approximation of a cell by boxes from the grid.  
        \begin{definition}[$\#$-Operator]
            For a set $t\subseteq [0,1]^d$ define $t^{\#}$ as the union of boxes $B \in G_n$ with midpoint in $t$.
        \end{definition}
            \begin{definition}
                For two sets $t' \subseteq t \subseteq [0,1]^d$ define \[ \bdiff(t,t'):= (t' \Delta t'^{\#}) \setminus ( t \Delta t^{\#}) ,\]where $\Delta$ denotes the symmetric difference between two sets, i.e. $A \Delta B = (A \setminus B) \cup (B\setminus A)$.
            \end{definition}
            We refer to \Cref{fig:oblique_difference_illustration} for an illustration.
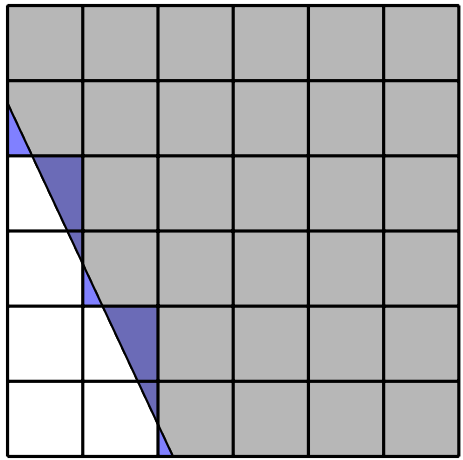
\begin{figure}
    \centering
    \begin{tikzpicture}[thick,centered]
    \coordinate (lowerleft) at (1,1);
    \coordinate (upperright) at (7,7);
    \coordinate (linestart) at (1,5.7);
    \coordinate (lineend) at (3.2,1);
    \draw[black,dashed] (linestart) -- (lineend);
    \coordinate (upperleft) at (1,7);
    \coordinate (lowerright) at (7,1);
    \begin{scope}
        \clip (1,5) rectangle (2,6);
        \draw[fill=blue!50] (linestart) -- (lineend) -- (lowerleft) -- cycle;
    \end{scope}

    \begin{scope}
        \clip (1,4) rectangle (2,5);
        \draw[fill=blue!50] (linestart) -- (lineend) -- (upperright) -- cycle;
    \end{scope}

    \begin{scope}
        \clip (1,3) rectangle (2,4);
        \draw[fill=blue!50] (linestart) -- (lineend) -- (upperright) -- cycle;
    \end{scope}

    \begin{scope}
        \clip (2,3) rectangle (3,4);
        \draw[fill=blue!50] (linestart) -- (lineend) -- (lowerleft) -- cycle;
    \end{scope}

    \begin{scope}
        \clip (2,2) rectangle (3,3);
        \draw[fill=blue!50] (linestart) -- (lineend) -- (upperright) -- cycle;
    \end{scope}

     \begin{scope}
        \clip (2,1) rectangle (3,2);
        \draw[fill=blue!50] (linestart) -- (lineend) -- (upperright) -- cycle;
    \end{scope}
        
    \begin{scope}
        \clip (3,1) rectangle (4,2);
        \draw[fill=blue!50] (linestart) -- (lineend) -- (lowerleft) -- cycle;
    \end{scope}

    \draw[fill=black!70,opacity=0.4] (upperleft) -- (upperright) -- (lowerright) -- (lineend) -- (linestart) -- cycle;    
    
    \draw [very thick, black, step=1.0cm] (1,1) grid +(6,6);

\end{tikzpicture}
    \caption{Illustration of $\rho(t,t')$ (blue parts) where $t=[0,1]^2$ and $t'$ is the area above the diagonal line.}
    \label{fig:oblique_difference_illustration}
\end{figure}
            \begin{condition}{(Tree-2)}\label{cond:T_boundary}
                There exist constants $C,B,\beta > 0$ and a class $\mathcal{H}$ of subsets $H \subseteq Q^d$ satisfying the following properties
                \begin{itemize}
                    \item[$\bullet$] for each $H\in \mathcal H$, $\# H \leq C g_n^{d-1}$,
                    \item[$\bullet$] for each $t\in\mathcal{C}$ and $t'\in \mathcal{C}_1(t)$ there exists $H\in\mathcal{H}$ such that \[\rho(t,t') \subseteq \bigcup_{\bq\in H} B_\bq,\]
                    \item[$\bullet$] $\# \mathcal H \leq Bn^{\beta}$.
                \end{itemize}
            \end{condition}
            Condition \ref{cond:T_boundary} is concerned with the approximation of a set $t$ by $t^{\#}$ in each iteration step. On the one hand, by the second and third point, the condition controls the complexity of possible deviations from a cell $t$ to its approximation $t^{\#}$ after a single step of the algorithm. On the other hand, the first and second property can also be seen as a condition on the boundaries of subsets constructed from the algorithm. Observe that $g_n^{-d}$ is the volume of a box from $G_n$ and $g_n^{-1}$ is the volume of a stripe $G(j,q)$. Therefore, we may say that the first two properties ensure that partitioning $t$ in one iteration step is a \quot{low-dimensional operation} on $t$.
            
    \subsection{Consistency}\label{subsec:main-consistency-result}
    Before stating our main theorem, we introduce three additional conditions on the distribution of $(X_i,Y_i)$. They are typical moment conditions and were also imposed by \citet{Chi}.
    \begin{condition}{(C2)}\label{cond:C2}
		$X_1$ has a Lebesgue-density $f$ that is bounded away from $0$ and $\infty$.
\end{condition}
\begin{condition}{(C3)}\label{cond:C3}
		The random variables $\varepsilon_1,\dots, \varepsilon_n, X_1,\dots,X_n, X$ are independent, where $X$ is distributed as $X_i$, and $d= O( n^{K_0})$ with some $K_0 > 0$. Furthermore, $\varepsilon_1,\dots \varepsilon_n$ are identically distributed where the distribution of $\varepsilon_i$ is symmetric around $0$ and $\E|\varepsilon_i|^q < \infty$ for some sufficiently large $q > 0$.
	\end{condition}
\begin{condition}{(C4)}\label{cond:C4}
		The regression function $m$ is bounded, i.e. \[M_0 :=  \sup_{x\in [0,1]^d}|m(x)| < \infty.\]
  \end{condition}
    \begin{theorem}[Consistency for general tree estimators]\label{thm:main3}
            Suppose that the probabilistic SID Condition \ref{cond:sid-general} holds with $\delta \geq 1-L^{-1}$. Furthermore, assume Conditions \ref{cond:C2}, \ref{cond:C3}, \ref{cond:C4} and \ref{cond:T_dim}, \ref{cond:T_boundary}. Then, there exists $c > 0$, such that for $k = k_n < c\log(n)$ with $k\to \infty$,
            \begin{align*}
                \lim_{n \to \infty} \E\left[ \left( \hat{m}_{\Tgeneral}(X) - m(X) \right)^2 \right]  = 0.
            \end{align*}
        \end{theorem}
        \begin{remark}
        We note that by the conditional Jensen inequality, the above result also implies
         \begin{align*}
                          \lim_{n\to\infty} \E\Big[ \big(m(X) -  \E[\hat{m}_{\hat{T}} (X) | D_n, X] \big)^2  \Big] = 0,
         \end{align*}
where, $D_n =(X_1,Y_1,\dots,X_n,Y_n)$. The conditional expectation can be seen as averaging over randomized trees grown on data $D_n$ and may thus be interpreted as a tree ensemble. 
     \end{remark}
 Proofs are postponed to the appendix, see also \Cref{remark:notes_proof} at the end of this section.
        \Cref{thm:main3} shows in large generality that tree based algorithms are consistent under probabilistic SID conditions. The following examples illustrate its wide-range applicability. Note that in order to show the applicability for Theorem \ref{thm:main3} to different algorithms, we only need to show Conditions \ref{cond:T_dim} and \ref{cond:T_boundary}.
        \begin{example}[Example \ref{ex:CART_2} continued]
            Let $\mathcal{C}$ be all rectangular subsets of $[0,1]^d$ and let $\mathcal{C}_1(t)$ all possible sets obtained from $t$ through a rectangular cut.
            \begin{itemize}
                \item Observe that there are at most $d(g_n+1)$ possibilities to separate the points in $G_n^{\#}$ through a rectangular cut. Thus, if $d$ is at most of polynomial order, Condition \ref{cond:T_dim} holds.
                \item Furthermore, if $\mathcal{H}$ is chosen as the collection of stripes $G(j,q)$ (more precisely, the grid points in $Q^d$ corresponding to $G(j,q)$), it is easy to see that \Cref{cond:T_boundary} is satisfied.
            \end{itemize}
            Thus, Theorem \ref{thm:main3} generalizes the results of \citet{Chi} when disregarding convergence rates.
        \end{example}
        \begin{example}[Extremely Randomized Trees]\label{ex:extratrees}
             Consider the Extremely Randomized Trees algorithm \citep[see][]{extra}. Given a rectangular cell $t$, a set of $\texttt{nsplit}$ candidate partitions $\mathcal{P}(t)$ into $L=2$ rectangles is given by $\texttt{nsplit}$ split points drawn uniformly at random, and splitting $t$ using rectangular cuts. Here, randomization can be formalized by independent random variables $\gamma \sim \operatorname{Unif}\{1,\dots,d\}$, $S \sim \operatorname{Unif}[0,1]$ (such that $\gamma$ determines the coordinate, and $S$ the proportion of the split point position along this coordinate). Then, $\mathcal{C}$ contains all rectangular subsets of $[0,1]^d$ and, as in the previous example, Conditions \ref{cond:T_dim} and \ref{cond:T_boundary} are fulfilled under the remaining assumptions of the theorem.
        \end{example}
            We emphasize that our consistency result for Extremely Randomized Trees allows for both random choices of split candidates while still using the CART-principle to evaluate these candidates. To the best of our knowledge, prior consistency results related to Extremely Randomized Trees have focused on either of the extreme cases, which are \texttt{nsplit = $\infty$} (i.e. resembling the CART algorithm), or \texttt{nsplit=1}. The latter case for $d=1$ is related to Purely Uniform Random Forests analyzed by \citet{genuer2012variance}.
        \begin{example}[Oblique Trees]\label{ex:oblique}
            Consider the special case of Oblique Regression Trees \citep{cartbook} where a cell $t$ is split into $L=2$ sets of the form $\{ x \in t : \langle b,x\rangle \leq s \}$, where $b\in \R^d$, $s \in [0,1]$ with $b_j = 0$ for $j\neq k_1,k_2$ for some $k_1,k_2$ and its complement. Thus, a set is cut by a hyperplane, not only perpendicular to an axis. By drawing random candidates for $b$, this gives rise to randomized partitions. If $d$ is of polynomial order, it can be verified that Conditions \ref{cond:T_dim} and \ref{cond:T_boundary} are fulfilled. A proof is provided in the appendix (see \Cref{subsec:proof_oblique}). For more general results on consistency of Oblique Trees, we refer to \citet{oblique_consistency_klusowski,oblique_consistency_zhan}.
        \end{example}

        \begin{example}[Interaction Forests]\label{ex:interaction_forests}
We can apply our theorem to Interaction Forest algorithm introduced by \citet{interactionforests}. Here, in each iteration step, a current set $t$ is split into two daughter cells of the form $t_1$, $t_2 = t\setminus t_1$ where, given $j_1,j_2 \in \{1,\dots,d\}$ and $c_1, c_2 \in t^{(j_1)}$ with $t^{(j)}= \{x_j: x\in t\}$, $t_1$ has one of the following forms
            \begin{itemize}
                \item $t_1 = \{ x \in t : x_{j_1} \blacklozenge_1 c_1 \text{ and }x_{j_2} \blacklozenge_2 c_2 \}$ with $\blacklozenge_1, \blacklozenge_2 \in \{ \leq, \geq \}$,
        \item $t_1 = \{ x \in t : x_{j_1} \leq c_1, x_{j_2} \leq c_2 \} \cup  \{ x \in t : x_{j_1} \geq c_1, x_{j_2} \geq c_2 \}$,
        \item $t_1 = \{ x \in t : x_{j_l} \leq c_l\}$, where $l=1,2$.
    \end{itemize}
    In \Cref{fig:intf} the different possibilities are illustrated (note that the first point accounts for 4, and the third for 2 pictures). Here, the splits are randomized, by drawing \texttt{npairs} combinations $(j_1, j_2, c_1, c_2)$ of coordinates and split points. This defines a randomized collection of partitions $\sD(t)$ consisting of $7 \times \texttt{npairs}$ partitions of $t$ of size $L=2$. Under the conditions of \Cref{thm:main3}, Interaction Forests are consistent. See \Cref{subsec:proof-main2} for a proof.
\end{example}
  \begin{figure}
  \centering
        \includegraphics[scale=0.35]{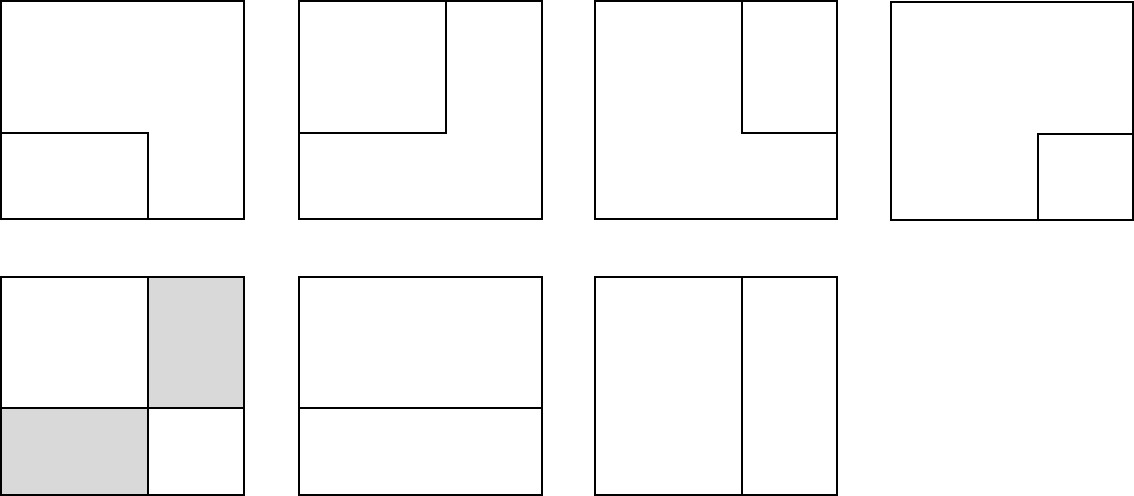}
        \caption{Illustration of candidate splits in Interaction Forests. Adapted from \citet[Figure 2]{interactionforests}}.
        \label{fig:intf}
    \end{figure}
Another example will be studied in \Cref{sec:consistency}. The class of regression functions for which the result holds is determined by the probabilistic SID Condition \ref{cond:sid-general}. Depending on the specific algorithm, the possible subsets $t\in \mathcal{C}$ that can be constructed from an algorithm may include complex shapes and thus, the probabilistic SID can become difficult to handle mathematically. See \Cref{fig:intf-partition} for an illustration of possible shapes for Interaction Forests. In the remainder, we develop a deeper understanding on the probabilistic SID condition for a specific algorithm arising naturally from Random Forests, and we show that this leads to a class of regression functions that is strictly larger than those satisfying SID for CART.
    \begin{figure}[h]
        \centering
        \includegraphics[width=0.5\linewidth]{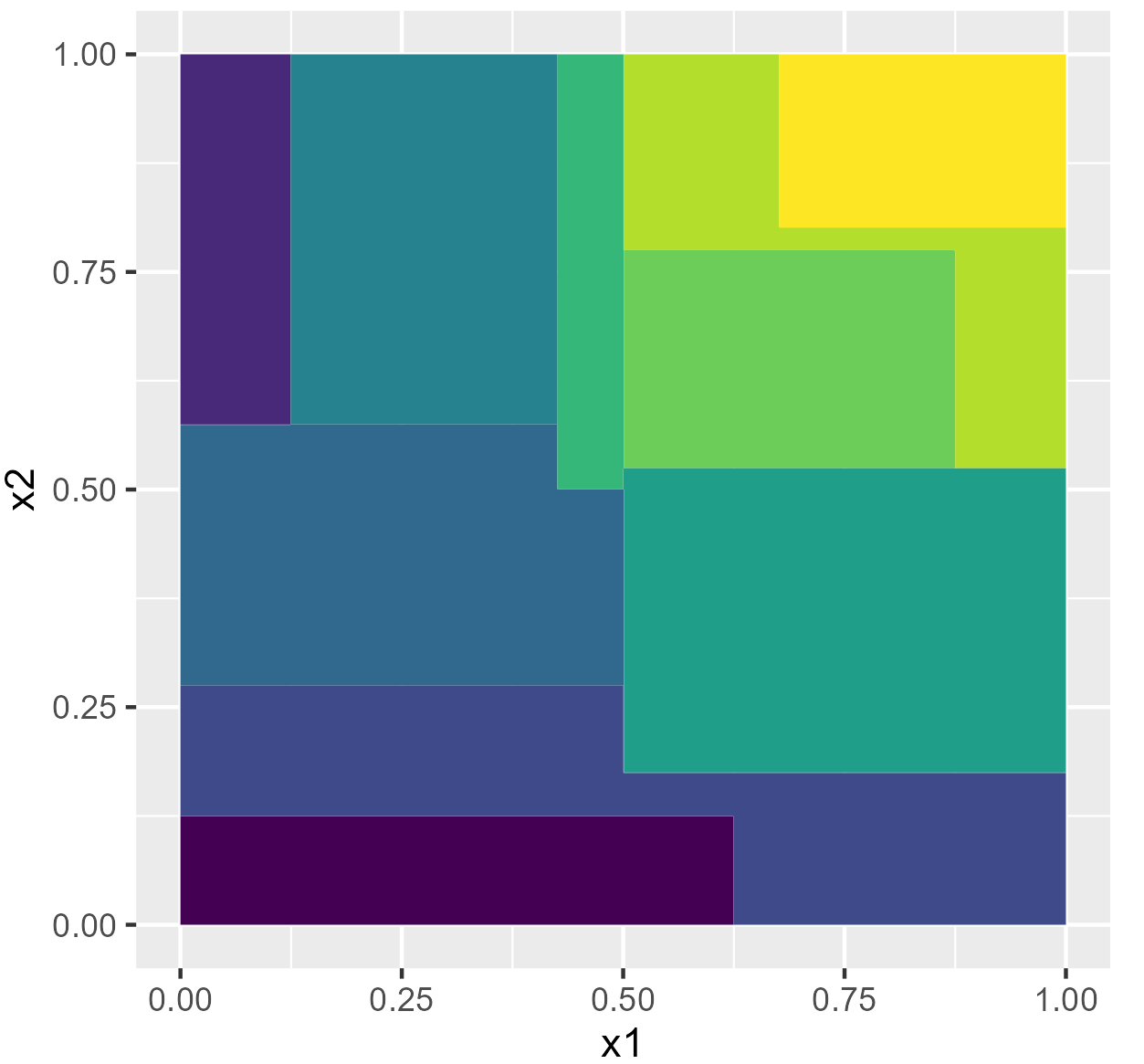}
        \caption{Example of a partition corresponding to a tree in Interaction Forests.}
        \label{fig:intf-partition}
    \end{figure}

\begin{remark}\label[remark]{remark:notes_proof}
An outline for the proof of \Cref{thm:main3} and the proofs itself are deferred to \Cref{sec:preparation_proof} and \Cref{sec:proofs} (\ref{subsec:proof_lemma_error1}-\ref{subsec:proof-main3}). We extend ideas of \citet{Chi}, but the arguments differ due to the random nature of \Cref{cond:sid-general}. Roughly speaking, $\hat{T}$ is compared to $T^*$ where the latter is the theoretical algorithm, i.e. $\widehat \zwei$ in \Cref{def:algo_theory} is replaced by $\zwei$, and to a suitably defined combination of $\hat{T}$ and $T^*$. Importantly, $T^*$ does not depend on data, but on the randomization variables. Note that, $T^*$ in the work of \citet{Chi} can also depend on randomization due to feature restriction in a CART split. In contrast to \citet{Chi}, in our setting, randomization is coupled to the probabilistic SID condition, whereas \citet{Chi} assume that the SID condition holds uniformly over all realizations. This requires a careful adaption of the arguments. Note that the impurity gain inequality from \Cref{cond:sid-general} is recursively used to control the squared error on population level. Here, due to randomness, we use that the inequality is valid sufficiently often during tree building, in each branch. Finally, granting Conditions \ref{cond:T_dim}, \ref{cond:T_boundary}, the supremum over all possible partitions generated through the algorithm can be reduced to a maximum, while at the same time deviations between population and empirical means (and related quantities) can be controlled.
\end{remark}

\section{Consistency for a Larger Function Class}\label{sec:consistency}
    In \Cref{thm:main3} we provided a general consistency result under suitable probabilistic SID conditions. The goal of this chapter is to give an example where the deterministic SID condition for CART (see Example \ref{ex:CART_2}) implies the probabilistic SID condition for another tree based algorithm - but not vice versa. To accomplish this, we consider an illustrative hybrid algorithm which combines ideas of classical regression trees with splits chosen randomly and independently of the observations. In \Cref{subsec:rsrf}, we introduce this algorithm coined Random Split Random Forests (RSRF). RSRF naturally extends the main principles in Random Forest. In particular, the constructed cells remain rectangular. In \Cref{sec:extending_function_class}, we will compare the probabilistic SID condition for RSRF with the SID condition imposed by \citet{Chi} for usual CART trees.
\subsection{Description of Random Split Random Forests}\label{subsec:rsrf}

\begin{figure}
\centering
\begin{minipage}{0.35\textwidth}
   \includegraphics[width=0.99\linewidth]{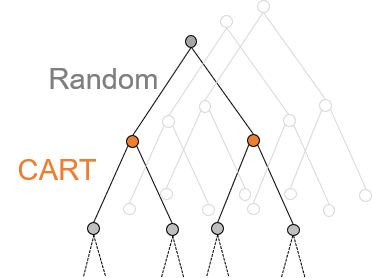}
  \captionof{figure}{Illustration of \rsrf. Background trees illustrate other possible candidate partitions.}
  \label{fig:RSRF}
\end{minipage}%
\hspace{0.5cm}
\begin{minipage}{.6\textwidth}\centering
        \resizebox{0.9 \linewidth}{0.3 \linewidth}{%
  \begin{tikzpicture}[thick,centered]
            	\node[rectangle,draw, minimum height=3cm, minimum width=4cm] (r) at (0,0) {};
			\coordinate (splitpoint)      at ($(r.north west) + 0.5*(r.south)$);
                \coordinate (splitpointend)   at ($(r.north east) + 0.5*(r.south)$) ;
                \coordinate (splitpoint02)    at ($(r.north) + 0.25*(r.east)$);
                \coordinate (splitpoint02end) at ($(splitpoint02) + 0.5*(r.south)$);
                \coordinate (splitpoint03)    at ($(r.south) + 0.2*(r.west)$);
                 \coordinate (splitpoint03end)    at ($(r.north) + 0.2*(r.west) + 0.5*(r.south)$);

				\draw (splitpoint02) -- (splitpoint02end);
			\draw (splitpoint) -- (splitpointend);
   			\draw (splitpoint03) -- (splitpoint03end);

			\node at ( $(r.north) - (0,0.4)$ ) {};
			\node at ( $(r.center) -(0,0.3)$ ) {};
  
        	\node[rectangle,draw, minimum height=3cm, minimum width=4cm] (t) at (-10,0) {};

                 \node[rectangle,draw, minimum height=3cm, minimum width=4cm] (t1) at (-5,0) {};
			\coordinate (split01)      at ($(t1.north west) + 0.5*(r.south)$);
                \coordinate (split01end)   at ($(t1.north east) + 0.5*(r.south)$) ;
				\draw (split01) -- (split01end);

                \draw[dashed] ([yshift=0.2cm, xshift=+0.3cm]t1.north) to[bend left = 40] node[midway, above]{\huge CART splits} ([yshift=0.2cm]r.north);
                \draw[dashed] ([yshift=0.2cm]t.north) to[bend left = 40]  node[midway, above]{\huge Random split}([yshift=0.2cm,  xshift=-0.3cm]t1.north);

                \node at (t.center) {\huge $t$};
                \node at ($0.5*(splitpoint) + 0.5*(splitpoint02)$) {\huge $t_{11}$};
                \node at ($0.5*(splitpointend) + 0.5*(splitpoint02)$) {\huge $t_{12}$};
                \node at ($0.5*(splitpoint) + 0.5*(splitpoint03)$) {\huge $t_{21}$};
                \node at ($0.5*(splitpointend) + 0.5*(splitpoint03)$) {\huge $t_{22}$};

		\end{tikzpicture}
 }

  
	\captionof{figure}{Illustration of the procedure used by RSRF for splitting a cell $t$ into $t_{11},t_{12},t_{21}, t_{22}$.}
 \label{figure:random-cart}
\end{minipage}
\end{figure}

We introduce the algorithm coined Random Split Random Forests. It depends on a width parameter $W\in\mathbb{N}$. Figures \ref{fig:RSRF} and \ref{figure:random-cart} serve as an illustration for \rsrf. For a rectangular set $t\subseteq[0,1]^d$, denote by $t^{(j)}$ the $j$-th side of $t$.\begin{definition}\label[definition]{def:algo_theory}
        Let $x_i \in [0,1]^d, y_i \in\R$, $i=1,\dots,n$, be given and let $k\in 2\mathbb{N}$. Starting with $T_0=\{\rootcell\}$ with $\rootcell = [0,1]^d$, for $m=0,2,4,\dots, k-2$ apply the following steps to all current leaf nodes $t \in T_{m}$.
        \begin{enumerate}
            \item Draw $\width$ many pairs $(j^w,c^w)\in \{1,\dots,d\} \times t^{(j)}$ uniformly at random.
            \item For each $w=1,\dots,W$ split $t$ at $(j^w,c^w)$ into $t^w_{1}$ and $t^w_{2}$ and then split $t^w_{1}$ and $t^w_{2}$ according to the Sample-CART criterion in \Cref{def:sample-cart}. This gives a partition of $t$ into $P^w = \{t^{w}_{1,1},t^{w}_{1,2},t^{w}_{2,1},t^{w}_{2,2} \}$
            \item Choose the splits with index $w_{\text{best}} \in \underset{w=1,\dots,W}{\arg\max}\ \widehat{\zwei}\left( t; P^w \right)$.
            \item Add $t_{j,l}^{w_{\text{best}}}$ to $T_{m+2}$ for $j,l\in\{1,2\}$.
        \end{enumerate}
        Set $\hat{T}:=T_{k}$.
    \end{definition}
    In this section, we denote by $\hat{m}_{\hat{T}}$ the tree estimator \eqref{eq:general_tree_estimator}, corresponding to the \rsrf\ algorithm from \Cref{def:algo_theory}. Due to the randomness induced, such trees may be aggregated to an ensemble of trees, following the same principles as Random Forests.
\subsection{Consistency for RSRF is Valid for a Larger Function Class}\label{sec:extending_function_class}
The tree growing procedure in \rsrf\ differs from Random Forests through the random splits and evaluation of splits using the impurity decrease $\zwei$ with $L=4$. 
We shall now discuss how this enlarges the class of functions for which consistency holds, compared to existing consistency results for regression trees and Random Forests.\\
Recall the definition of impurity decrease from \Cref{def:impurity_decrease}. We say that a cell $t$ is split (into $t_1$ and $t_2$) at a split point $\bc = (j,c)\in\{1,\dots, d\}\times [0,1]$, if $t_1 = \{x \in t: x_j \leq c\}$ and $t_2 = t \setminus t_1$. The following convenient notation is best described in words (see \Cref{def:cut-notation}\label{def-in-words:cut-notation} for a formal definition). We denote by $\zwei_t\left( \bc | \bc_1,\bc_2 \right)$ the impurity decrease for the partition that is obtained from $t$ by splitting it at a split point $\bc =(j,c) \in \{1,\dots, d\} \times [0,1]$, after which the resulting two daughter cells are each split again at some split points $\bc_1$ and $\bc_2$. We set its value to $-\infty$ in case that one of the splits is placed outside of the cell.
\Cref{table:summary_impurity_decrease_notation} summarizes all the notation used for impurity decrease.\\
     For RSRF, we introduce a simplified version of \Cref{cond:sid-general}, where we omit the technicality $\delta \geq 1-L^{-1}$. It will help in the proofs to come. 
    \begin{table}[t]\renewcommand\arraystretch{1.3}
    \centering
    \begin{tabular}{p{2cm}|p{10cm}}
        $\zwei(t;P)$ & Impurity decrease given a partition $P$ of $t$. See Definition \ref{def:impurity_decrease}.  \\
        \hline
        $\zwei_t( \bc )$ & Impurity decrease for the partition of $t$ obtained through a single rectangular cut at $\bc =(j,c) \in \{1,\dots,d\}\times t^{(j)}$. See Definition \ref{def:cut-notation_1step}. \\
        \hline
        $\zwei_t(\bc|\bc_1,\bc_2) $ & Impurity decrease when splitting $t$ first at $\bc$, and then splitting its daughters at $\bc_1$ and $\bc_2$. See Definition \ref{def:cut-notation}.
    \end{tabular}
    \caption{Notations related to impurity decrease.}
    \label{table:summary_impurity_decrease_notation}
\end{table}

\begin{condition}{(C1-\rsrf)}[Probabilistic SID for \rsrf]\label{cond:C1}
    There exist $\alpha_1 \geq 1$, $\delta \in (0,1]$ such that for any rectangular cell $t =\bigtimes_{j=1}^d t^{(j)}$ we have:
	\begin{align}\label{eq:C1-good}
		\var( m(X) | X\in t) \leq \alpha_1 \sup_{\bc_1, \bc_2} \zwei_t \left( \bc | \bc_1, \bc_2  \right)  \text{ with probability at least } \delta.
	\end{align}
	Here, the daughter cells of $t$ are obtained from splitting $t$ at $\bc = (\gamma,U)$ where $\gamma$ is a random variable with $\PP( \gamma = j ) = \frac{1}{d}$, $j \in \{1,\dots, d\}$ and, conditionally on $\gamma =j$, $U$ is uniform on $t^{(j)}$.
\end{condition}
Clearly, the condition mimics the splitting scheme used in RSRF for a single random split. Our strategy to compare Condition \ref{cond:sid-general} for RSRF to the SID condition of \citet{Chi} is as follows. In \Cref{lemma:C1-C1'} we show that Condition \ref{cond:C1} implies \Cref{cond:sid-general} for RSRF for large $W$. Then, it is sufficient to show that a particular regression function $m$ satisfies \Cref{cond:C1} while violating the SID condition from \citet{Chi}. Note that RSRF also satisfies Conditions \ref{cond:T_dim} and \ref{cond:T_boundary}. Thus, we obtain the following corollary.
   \begin{corollary}[Consistency for\ \rsrf]\label[corollary]{thm:main}
	Let $\hat{T}$ be the tree building procedure from \Cref{def:algo_theory}. Assume \Cref{cond:C1} holds with some $\delta>0$ as well as Conditions \ref{cond:C2}, \ref{cond:C3}, \ref{cond:C4}. Assume that the width parameter $W$ is chosen such that $W  \geq \frac{-2\log(2)}{\log (1-\delta)}$.
        Then, for each sequence $k = k_n$ with $k\to \infty$ and $k < \frac{1}{4} \log_2( n) $, it holds:
		\begin{align*}
   			\lim_{n\to\infty} \E\Big[ \big(m(X) -  \hat{m}_{\hat{T}} (X) \big)^2 \Big] = 0.
		\end{align*}
\end{corollary}




We rephrase the SID condition from \citet{Chi}. Let us write $\zwei_t(\bc)$ for the corresponding impurity decrease $\zweiarg{2}(t;t_1,t_2)$ when using the single split $\bc =(j,c) \in \{1,\dots,d\}\times [0,1]$ (we set $\zwei_t(\bc)$ to $-\infty$ in case of splits outside of the cell $t$).\label{def-in-words:cut-notation_1step} See also \Cref{def:cut-notation_1step}. 
\begin{definition}[SID condition for CART; \citeauthor{Chi}, \citeyear{Chi}]\label[definition]{def:C1-yingying}
    There exists some $\alpha_1\geq 1$ such that for each cell $t$, $\var( m(X) | X \in  t) \leq \alpha_{1} \sup_{\bc}\zwei_t(\bc )$. Here, the supremum is over $\bc \in \{1,\dots, d\} \times [0,1]$.
	\end{definition}
Consider the function $m:[0,1]^3 \to \R$, $m(x_1,x_2,x_3) = (x_1-0.5)(x_2-0.5)+x_3$, see also \Cref{fig:example_function}. As mentioned by \citet{Chi}, this does not satisfy the above SID condition which follows from the first part of \Cref{prop:example}. Importantly, the probabilistic SID Condition \ref{cond:C1} is fulfilled.
 \begin{figure}[h]
    \centering
       \includegraphics[scale=0.25]{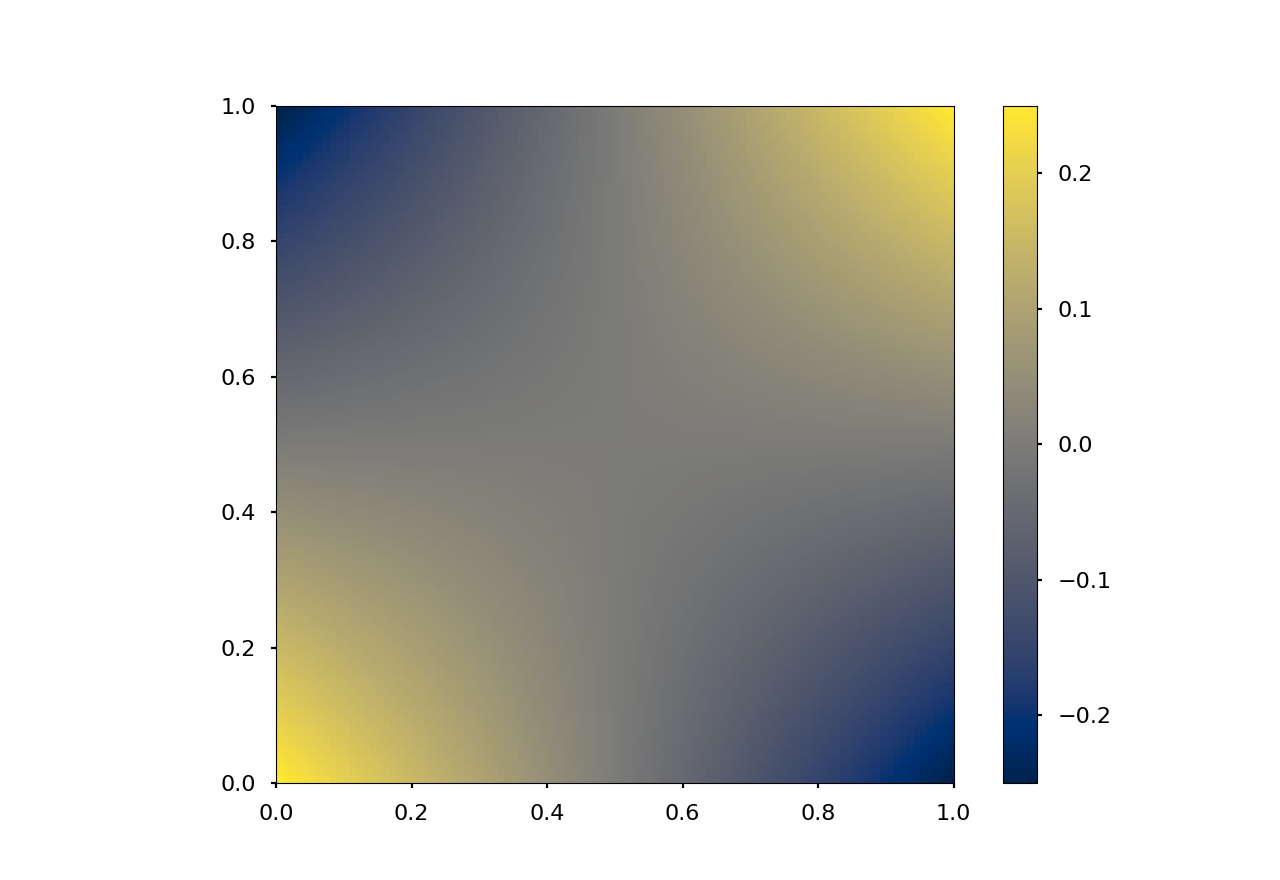}
       \caption{Plot of the function $g: [0,1]^2 \to \R$, $g(x_1,x_2) = (x_1-0.5)(x_2-0.5)$.}\label{fig:example_function}
    \end{figure}


\begin{proposition}\label[proposition]{prop:example}
Suppose $X$ is uniformly distributed on $[0,1]^3$ and let $m(x_1,x_2,x_3) = (x_1-0.5)(x_2-0.5)+x_3$. We have the following two statements.
\begin{enumerate}
    \item\label{itemprop:exSID} Let $0 \leq c_1 < c_2 \leq 1$ and $t = [0,1]^2\times [c_1,c_2]$. Then, $\var(m(X) | X \in t) \geq \frac{1}{144}$ and $\sup_{\bc}\zwei_t(\bc) = \frac{c_1+c_2}{2} \to 0$ as $c_1,c_2 \to 0$.
    \item\label{itemprop:exPSID} For any $\delta<\frac{2}{3}$ there exists a constant $\alpha_1=\alpha_1(\delta)$ such that \Cref{cond:C1} is satisfied.
\end{enumerate}
\end{proposition}


The proof of \Cref{prop:example} can be found in \Cref{sec:proof_example}. The main difference in assertion \ref{itemprop:exPSID} compared to \ref{itemprop:exSID} is that any \quot{symmetric} cell of the form $t = [0.5-l,0.5+l] \times [0.5-m, 0.5+m] \times [c_1,c_2]$ relates to large impurity decrease when $t$ is split into four cells both at coordinates $1$ and $2$. More precisely, suppose $t_1$ and $t_2$ are obtained by splitting $t$ at a point $0.5+a$, $a\in [-l,l]$ in the first coordinate. Then, split both daughter cells using a (theoretical) CART split at the second coordinate. By \Cref{remark:2step-1step},
\begin{align*}
    S:= &\sup_{\bc_1, \bc_2} \zwei_t \left( (1,0.5+\alpha)| \bc_1,\bc_2\right) \\
    &\geq 
     \PP(X\in t_1|X\in t) \sup_{c \in [0,1] }\zwei_{t_1}\left( (2,c) \right) \\
     &\quad + \PP(X\in t_2|X\in t) \sup_{c \in [0,1] }\zwei_{t_2}\left(  (2,c) \right) + \zwei(t; t_1, t_2).
\end{align*}
It can be checked that the last summand on the right side equals zero. Calculations then reveal that for $a = \kappa \cdot l$ with $\kappa \in (-1,1)$ we have 
\begin{equation}\label{eq:example_inequality_symmetric}
    \var(m(X)|X\in t)\leq \frac{16}{9(1-\kappa^2)}\times S.
\end{equation}
The above provides an intuition for the proof of \Cref{prop:example} \ref{itemprop:exPSID}. In the proof, this is generalized in order to provide such a bound for arbitrary cells. We emphasize that, above, $\kappa$ relates to the probability $\delta$ in the definition of our Condition \ref{cond:C1}. It is intuitive that, for larger values of $\delta$ we may need a larger $\alpha_1$ for \ref{cond:C1} to hold. In the present example, this connection is clear from \eqref{eq:example_inequality_symmetric}.\\
The proofs of the statements from \Cref{prop:example} which motivate the greater generality of our probabilistic SID condition along an exemplary model, are rather complex. We believe that the proof ideas can be transferred to other models.\\
Finally, let us note that the condition from \Cref{def:C1-yingying} indeed implies our Condition \ref{cond:C1}. This follows easily from \Cref{remark:2step-1step} in the appendix. To wrap this up, our result from \Cref{thm:main} is valid for a strictly larger class of regression functions than the consistency result of \citet{Chi} for CART trees.


\section{Conclusion and Outlook}\label{sec:outlook}
We developed theoretical guarantees for a general class of tree estimators, by extending recent developments from \citet{Chi} on Random Forest consistency. By our new probabilistic SID condition, we account for the usage of additional randomness in the tree building - which is common in many tree based algorithms such as Extremely Randomized Trees. In addition, for the \rsrf\ algorithm, we argued that the class of regression functions covered by our general result is strictly larger than the ones covered by recent results on Random Forests. This was proven by considering a regression function with a so-called \emph{pure interaction} between two covariates. Such pure interactions are hard to detect for algorithms using CART because there is no one-dimensional marginal effect to guide to the pure interaction effect \citep[compare][]{wright2016}). However, it remains an open problem to decide whether or not Random Forest (using the CART criterion) is inconsistent in pure interaction scenarios. In future work, we intend to further investigate the disparity between Random Forests and related modifications such as \rsrf\ in a simulation study.



\newpage

\appendix
\section{Structure of the Appendix}
The supplement is structured as follows. In \Cref{sec:preparation_proof}, we provide an overview of the proof of the main consistency results from \Cref{thm:main3} and \Cref{thm:main}. Furthermore, necessary notations are introduced. In \Cref{subsec:proof_lemma_error1,subsec:C1-related,subsec:proof_C5-fulfilled_thm3,subsec:proof-main3}, the remaining details to the proofs concerning \Cref{thm:main} and \Cref{thm:main3} can be found. \Cref{sec:proof_example} contains the proofs to \Cref{prop:example} from \Cref{sec:extending_function_class} in the main text. Finally, in \Cref{subsec:proof_oblique} we check Conditions \ref{cond:T_dim} and \ref{cond:T_boundary} for Oblique Trees and Interaction Forests (see \Cref{ex:interaction_forests,ex:oblique} in the main text). \Cref{sec:split_determining_sequences} contains a technical definition needed in the proofs. 
\section{Preparation and Sketch of the Proof for \Cref{thm:main3} and \Cref{thm:main}}\label{sec:preparation_proof}
For ease of notation we provide a detailed direct proof for \Cref{thm:main}. The proof for \Cref{thm:main3} is widely analogous and we refer to \Cref{subsec:proof-main3} for details on the necessary adaptions.\\
In the proof, we make use of the methods developed by \citet{Chi}. The essence in our theory is that the impurity gain \eqref{eq:C1-good} is only assumed to hold with probability $\delta$, and not with probability $1$. In our theory, we essentially show that the arguments of \citet{Chi} can still be used when they are carefully and appropriately adjusted. 
 Apart from this, in the proofs, we try to stay close to the arguments of \citet{Chi}. Here, additional considerations become necessary in parts of the proofs where impurity gain conditions are used as difficulties arise when the impurity condition becomes probabilistic. In particular, this concerns the proof of \Cref{lemma:error1}, stated below in \Cref{subsec:notes_proof}, which together with \Cref{lemma:error2} implies the statement of \Cref{thm:main}. Whereas \Cref{lemma:error2} immediately follows from \citet{Chi}, the proof of \Cref{lemma:error1} which gives a bound on the \quot{bias term} needs essentially new arguments. For the proof of 
\Cref{lemma:error1} we introduce two properties \ref{cond:C5'} and \ref{cond:C6'} which are shown to hold for \rsrf. For the proof of  \ref{cond:C5'} one can follow lines in the argumentation of \citet{Chi}, see \Cref{thm:C5-fulfilled}. But the proof of \ref{cond:C6'} needs more care and new arguments, see Lemmas \ref{lemma:C1-C1'} -\ref{lemma:C1impliesC6}. Given the proof of \ref{cond:C5'} and \ref{cond:C6'} one uses 
\Cref{thm3} to get \Cref{lemma:error1}. \Cref{thm3} is related to Theorem 3 of \citet{Chi} but the proof differs because now the new properties \ref{cond:C5'} and \ref{cond:C6'} are used with their now probabilistic formulation.

\subsection{Tree Notation and Partitions}\label{sec:RF}
In this section, we introduce notation required in the remainder of the paper. Furthermore, properties related to the impurity decrease are collected. 
    
 \subsubsection{Tree Notation}\label{sec:treenotation}
We introduce some notation for binary trees. This notation will be used in the proof of \Cref{thm:main}. Given a cell $t \subseteq[0,1]^d$, we use the notation $t_1$ and $t_2$, respectively, for the left and right daughter cell of $t$, in particular we have $t=t_1 \cup t_2$. Furthermore, we write $t_{1,1}$ and $t_{1,2}$ for the daughter cells of $t_1$, resp. $t_{2,1}$ and $t_{2,2}$ for the daughther cells of $t_{2}$. Then, $\{t_{1,1}, t_{1,2}, t_{2,1}, t_{2,2} \}$ forms a partition of $t$. This notation extends to $t_a = t_{a_1 \dots a_k}$	where $a = (a_1, \dots, a_k) \in \{1,2\}^k$ thus uniquely determines a path in a binary tree of depth $k$ starting at some (root) cell $t$ to one of the tree leafs $t_{a}$. To cover the case $k=0$, we use the convention $t_{\emptyset}=t$.\\ \ \\
	Given a binary tree and some cell at depth $l$, we may also want to identify this cell with the sequence of cells leading to it.
 That is, for a cell $t_{a}$ with given $a \in \{1,2\}^l$, let $\bt_{a}  :=  (t, t_{a_1} , t_{a_1a_2}, \dots, t_{a_1 \dots a_{l}}) $
be the tree branch corresponding to $a$. Here, $t$ is the root of the tree. 
    A binary tree $T$ of depth $k$ can thus be understood as the set of paths of cells, i.e.
	\begin{align*}
		T = \big\lbrace \bt_{a}  \ : a \in \{1,2\}^k \big\rbrace.
	\end{align*}
 Given a tree $T$ we denote by $\leafs{T}$ the set of all leafs of $T$. \Cref{fig:illustration_tre_notation} illustrates our notation.
     \begin{figure}\centering
         \begin{tikzpicture}[level distance=1.5cm,
    level 1/.style={sibling distance=3.5cm},
    level 2/.style={sibling distance=1.5cm},
    bold/.style={rectangle,solid,ultra thick} ]
\tikzstyle{every node}=[rectangle,draw]

    \node[bold] (Root) {$t$}
    child {
        node[bold] {$t_1$} 
        child { node {$t_{11}$} }
        child { node[bold] {$t_{12}$} }
    }
    child {
        node {$t_2$}
        child { node {$t_{21}$} }
        child { node {$t_{22}$} }
    };
\end{tikzpicture}
         \caption{Illustration of a tree of depth $2$ with root $t$. The boldface nodes form the branch $\bt_{12} = (t,t_1,t_{12})$. The leafs are given by $\leafs{T} = \{ t_{11},t_{12},t_{21}, t_{22} \}$.}
         \label{fig:illustration_tre_notation}
     \end{figure}
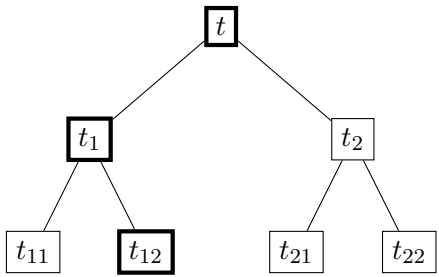
     Let $t$ be a cell partitioned into two cells $t_1$ and $t_2$. We occasionally use the bullet-notation $t_{\bullet}$ for the partition $\{t_{1}, t_{2}\}$. Similarly, if $t$ is instead partitioned into cells $t_{1,1},t_{1,2},t_{2,1},t_{2,2}$, we write $t_{\bullet \bullet}$ for $\{t_{a}:a \in \{1,2\}^2\}$.

\subsubsection{Further Notes on the Impurity Decrease}
Below, we list some results which will be of use in the proofs. Suppose a cell $t$ is split into two daugther cells which are then partitioned into grand-daughter cells. Refering to the corresponding decrease as \quot{$2$-step} impurity decrease, \Cref{remark:2step-1step} describes the $2-$step impurity decrease for the resulting partition in terms of the impurity decreases between $t$ and its daughters, and between its daughters and grand-daughters. Below, for clarification, we write $\zweiargneu{L}$ instead of $\zwei$ for the decrease associated to a partition of size $L$. 
	\begin{proposition}\label[proposition]{remark:2step-1step}
		 Let $t$ be a cell, $t_1, t_2$ daughter cells and $t_{j,1}, t_{j,2} \subseteq t_{j}$, $j = 1,2$ its grand-daughter cells. It holds that
		\begin{align*}
			\zweiargneu{4} \left( t; t_{\pp} \right) = \PP(X\in t_1| X \in t )\zweiargneu{2}\left( t_1;t_{1\bullet} \right) + \PP(X \in t_2 | X \in t )\zweiargneu{2}\left( t_2;  t_{2\bullet} \right) + \zweiargneu{2}\left( t; t_\bullet \right).
		\end{align*}
	\end{proposition}
		\begin{proofname}{Proof of \Cref{remark:2step-1step}}\label{remark:2step-1step-proof}
	We write $\PP(t_1|t) := \PP(X\in t_1 | X \in t)$. Recall that $\mu(t_1) = \E(m(X) | X \in t_1)$.
		It holds that
		\begin{align}\label{eq:II}
				\zwei \left( t; t_{\pp} \right)  = \sum_{\substack{j=1,2 \\ k=1,2}} \PP(t_{j,k} | t) [\mu(t_{j,k}) - \mu(t)]^2.
		\end{align}
		For $k=1,2$ it holds that
		\begin{align*}
			\PP(t_{j,k} | t )[ \mu(t_{j,k}) - \mu(t) ]^2 &= \PP(t_{j,k} | t_j) \PP( t_j | t ) \big\lbrace [\mu(t_{j,k}) - \mu(t_j)] + [\mu(t_j) - \mu(t)] \big\rbrace^2 \\
			&= \PP(t_j | t)  \PP(t_{j,k}| t_j) [\mu(t_{j,k}) - \mu(t_j)]^2 \\ & \qquad + \PP(t_{j,k}|t_j) \PP(t_j|t) [\mu(t_j) - \mu(t)]^2 \\ &\qquad + 2 \PP(t_{j,k}|t_j)\PP(t_j| t) \big( \mu(t_{j,k}) - \mu(t_j) \big)\big(\mu(t_j) - \mu(t)\big)
		\end{align*}
		Fix $j \in \{1,2\}$. The two summands with index $(j,1)$ and $(j,2)$ in \eqref{eq:II} thus sum up to
		\begin{align}
			\notag \PP(t_{j,1} | t) &[\mu(t_{j,1}) - \mu(t)]^2 + \PP(t_{j,2} | t) [\mu(t_{j,2}) - \mu(t)]^2 \\
			&\notag =\PP(t_j| t) \zweiargneu{2}\left(t_j; t_{j,1}, t_{j,2} \right)  + \underbrace{\big[ \PP(t_{j,1}|t_j) + \PP(t_{j,2} | t_j) \big]}_{=1} \PP(t_j | t) \big\lbrace \mu(t_j) - \mu(t) \big\rbrace^2 + \\
			& \notag\qquad 2[\mu(t_j) - \mu(t)] \PP(t_j|t) \big\lbrace (\mu(t_{j,1}) - \mu(t_j) )\PP(t_{j,1} | t_j) + (\mu(t_{j,2}) - \mu(t_j) ) \PP(t_{j,2}| t_j) \big\rbrace \\
			&\label{eq:inter}=\PP(t_j|t) \zweiargneu{2}\left(t_j; t_{j,1}, t_{j,2} \right) + \PP(t_j|t) (\mu(t_j) - \mu(t))^2,
		\end{align}
		since
		\begin{align*}
			(\mu(t_{j,1}) - \mu(t_j) )&\PP(t_{j,1} | t_j) + (\mu(t_{j,2}) - \mu(t_j) ) \PP(t_{j,2}| t_j)\\
			&= \underbrace{\mu(t_{j,1}) \PP(t_{j,1} | t_j) + \mu(t_{j,2})\PP(t_{j,2}|t_j)}_{=\mu(t_j)} - \mu(t_j)(\underbrace{\PP(t_{j,1}|t_j) + \PP(t_{j,2}|t_j)}_{=1}) \\
			&= 0.
		\end{align*}
		Finally, we obtain the desired result by summation of \eqref{eq:inter} for $j=1$ and $j=2$.
	\end{proofname}
For a proof of the next result, see \citet[p.~273]{cartbook}. 
        \begin{proposition}
        The impurity decrease for a partition into two cells can also be calculated through the following relation. 
            \begin{align}\label{eq:1step-alternative}
			\zwei\left(t; t_1, t_2\right) = \PP(X \in t_1| X \in t)\PP(X \in t_2|X \in t) \big[ \mu(t_1) - \mu(t_2) \big]^2.
		\end{align}
        \end{proposition}
        Below, we state a closed expression for the quantity $\eins$ defined in \Cref{def:impurity_decrease}. For the sake of completeness, we provide a proof.
\begin{remark}\label[remark]{remark:eins_plus_zwei}
Let $P=\{t_1,\dots, t_L\}$ be a partition of $t$. It holds that
     \begin{align*}
         \eins\left(t;P\right) 
        = \sum_{l=1}^L \PP( X \in t_l| X\in t)\var(m(X) | X \in t_l).
    \end{align*}
\end{remark}


\begin{proofname}{Proof of \Cref{remark:eins_plus_zwei}}
We need to show that the right hand side of the equation equals $\var(m(X)|X\in t) - \zwei(t;P)$. It holds that
    \begin{align*}
        \var(m (X)|X\in t)
        =  \sum_{l=1}^L \E\left[  \1_{(X \in t_l)} \left( m(X) - \E[m(X) | X\in t] \right)^2 \big |X \in t \right].
    \end{align*}
    Note that the conditional expectation inside the sum is equal to
    \begin{align}\label{eq:split_eins_plus_zwei}
    \begin{split}
        &\E \left[  \1_{(X \in t_l)} \left( m(X) - \E[m(X) | X\in t_l] \right)^2 \big| X \in t \right] \\
         + &\E\left[  \1_{(X\in t_l)} \left( \E[m(X) | X\in t_l] - \E[m(X) | X\in t] \right)^2 | X\in t \right] \\
           +&2\E\left[  \1_{(X\in t_l)}\left( m(X) - \E[m(X) | X\in t_l] \right) \left( \E[m(X)|X\in t_l] - \E[m(X)| X \in t] \right)  \big|X \in t \right]
    \end{split}
    \end{align}
    It holds the last summand in \eqref{eq:split_eins_plus_zwei} is zero, as
    \begin{align*}
        &\E\left[  \1_{(X\in t_l)}\left( m(X) - \E[m(X) | X\in t_l] \right) \big| X \in t \right]
       \\ & \hspace{2cm}= \PP(X\in t_l |X \in t) \E\left[ \left( m(X) - \E[m(X) | X\in t_l] \big| X \in t_l \right)\right]= 0.
    \end{align*}
    The first summand in \eqref{eq:split_eins_plus_zwei} equals $\PP(X\in t_l|X\in t)\var(m(X)| X\in t_l)$, and the second one equals $\PP(X\in t_l |X\in t)\left(\E[m(X)|X\in t_l] - \E[m(X)| X \in t]\right)^2$. The claim follows by definition of $\zwei(t;P)$.
\end{proofname}


\subsubsection{Formal Definition for the Cut Notations}
Below, we give formal definitions of $\zwei_{t}(\bc|\bc_1,\bc_2)$ introduced in \Cref{sec:consistency} (see pages \pageref{def-in-words:cut-notation} and \pageref{def-in-words:cut-notation_1step}). We say that a cell $t \subseteq[0,1]^d$ is rectangular if it is of the form $t = \bigtimes_{j=1}^d t^{(j)} \subseteq [0,1]^d$ with intervals $t^{(j)} \subseteq [0,1]$. 
    \begin{definition}\label[definition]{def:cut-notation}
        Given a rectangular subset $t \subseteq [0,1]^d$ and pairs $\bc = (j,c), \bc_1 = (j_1,c_1), \bc_2 = (j_2,c_2) \in \{1,\dots,d \} \times [0,1]$ denote by $(t^{\bc,\bc_1,\bc_2}_{j,k})_{j,k = 1,2}$ the partition of $t$ obtained from the following (consecutive) splits:
        \begin{enumerate}
            \item Split $t$ at split point $\bc$ into $t_1 = \{ x \in t | x_j \leq c\}$ and $t_2 = t\setminus t_1$.
            \item Then, split $t_1$ and $t_2$ from (a) at split points $\bc_1$ and $\bc_2$, respectively, i.e.
            \begin{itemize}
                \item Let $t_{1,1} = \{ x \in t_1 | x_{j_1} \leq c_1\}$ and $t_{1,2} = t_1 \setminus t_{1,1}$
            \item Let $t_{2,1} = \{ x \in t_2 | x_{j_2} \leq c_2\}$ and $t_{2,2} = t_2 \setminus t_{2,1}$.
            \end{itemize}
        \end{enumerate}
        Then, we define
        \begin{align*}
        \zwei_t\left( \bc | \bc_1,\bc_2 \right) := \begin{cases}
             -\infty,   & \ \text{if some $t^{\bc, \bc_1,\bc_2}_{j,k}$ is empty,} \\
            \zwei\left( t; (t^{\bc,\bc_1,\bc_2}_{j,k})_{j,k \in \{1,2\} } \right), & \text{ otherwise.}
        \end{cases}
    \end{align*}
    \end{definition}

     \begin{definition}\label[definition]{def:cut-notation_1step}
        Given a rectangular subset $t \subseteq [0,1]^d$ and a split point $\bc = (j,c) \in \{1,\dots,d \} \times [0,1]^d$ denote by $t^{\bc}_{1}= \{ x \in t | x_j \leq c\}$ and by $t^{\bc}_2 = t \setminus t^{\bc}_1$.
        Then, we define
        \begin{align*}
        \zwei_t( \bc ) := \begin{cases}
             -\infty,   & \ \text{if some $t^{\bc}_{j}$ is empty,} \\
             \zweiarg{2}\left( t; (t^{\bc}_{j})_{j\in\{1,2\}} \right), & \text{ otherwise.}
        \end{cases}
    \end{align*}
    \end{definition}

    \subsection{Tree Growing Rules and Split Determining Sequences}\label{subsec:treegrowingrules}
    To inherit the algorithmic structure for RSRF into our notation, we borrow the term \emph{tree growing rule} from \citet{Chi}. A tree growing rule $T$ is always associated with a (deterministic) splitting criterion and, given $k\in 2\N$, it outputs the tree $T$ obtained by growing all cells up to level $k$ (starting from the root cell). The name $T$ is thus used for both the tree growing rule and the tree obtained by growing according to this rule. 
    For instance, given values $(x_i,y_i), i=1,\dots,n$, the Sample-CART-criterion forms a deterministic splitting criterion. The tree $\hat{T}$ from \Cref{def:algo_theory} can be seen as a tree growing rule when $(x_i,y_i), i=1,\dots,n$ and realizations of the random splits are given. Given a tree growing rule we can associate the corresponding estimator to it as well as its population version.
    \begin{definition}
	Let $x \in [0,1]^d$, $(x_i,y_i) \in [0,1]^d \times \R$, $i=1,\dots n$. Let $T$ be a tree growing rule.
	\begin{align*}
		\hat{m}_T(x) &:= \hat{m}_T( (x_i,y_i)_{i=1}^n,x) = \sum_{t \in \leafs{T} } \Big( \frac{1}{\card{t}} \sum_{ {i : x_i \in t} } y_i \Big) \1_{x\in t}&& \quot{Sample version}, \\
		m^*_T(x) &:= \sum_{t \in \leafs{T}} \1_{x\in t} \E( m(X) | X \in t) && \quot{Population version}.
	\end{align*}
    \end{definition}
     Note that $m^*_T$ may still depend on $(x_i,y_i)$, $i=1,\dots,n$, through the tree growing rule $T$. \\
     We need to formalize the random splits used in \rsrf. To do so, let us think of a binary tree of level $k$, where at each level $l=0,2,\dots, k-2$, any node is assigned with $W$ values $r_1,\dots, r_W \in \{1,\dots,d\} \times [0,1]$. Assume we have given such a value $r_w = (\gamma, u)$ and a rectangular cell $t$. Then, $t$ can be split at coordinate $\gamma$ and split point $t_{\gamma,\text{min}} + u( t_{\gamma,\text{max}} - t_{\gamma,\text{min}})$, 
    where $t_{\gamma,\text{min}}$ and $t_{\gamma,\text{max}}$ denote the end points of the $\gamma$th side $t^{(\gamma)}$ of $t$. The collection $\mathcal R_k$ of all such values (indexed via tree depth $l$, $j=1,\dots, 2^l$ and $w$) is called \emph{split determining sequence}. To summarize, a split determining sequence assigns nodes of a tree with $W$ possible split points. When $T$ is a tree growing rule, we say that it is \emph{associated} with the split determining sequence $\mathcal R_k$ if, at any second level for any node, a split is chosen from the $W$ possible split points corresponding to this node.\\    
    Suppose now the values $r_w$ above are instead random variables $R_w = (\gamma^w, U^w)$ with $\gamma^w$ uniform on $\{1,\dots d\}$ and (independently) $U^{w} \sim \text{Unif}[0,1]$. Furthermore, assume they are all independent (over $w$ and all nodes) and denote their collection by $\bR_k$. This is called a \emph{random split determining sequence}. Then, we say that a tree growing rule is associated with $\bR_k$ if it is associated with it on realization level.\\
    Now, let $\bR_k$ be such a random split determining sequence and assume that the family $\bR_k$ is independent of $D_n := (X_1,Y_1,\dots, X_n,Y_n)$. Clearly, we can regard the RSRF tree $\hat{T}$ as being associated with $\bR_k$. The RSRF estimator $\hat{m}_{\hat{T}}(x)$ then rewrites \[ \hat{m}_{\hat{T}}(x) =\hat{m}_{\hat{T}}(x, D_n, \bR_k).\]
    We note that we defined the terminology (``split determining sequence'', ``associated'') for RSRF more formally. Since we find this notationally demanding, we postpone this to \Cref{sec:split_determining_sequences}.
  
 \subsection{Outline of the Proof of \Cref{thm:main}}
Observe that
\begin{align*}
    \E\Big[ \big(m(X) -  & \hat{m}_{\hat{T}}(\bR_k, D_n, X) \big)^2 \Big]
    \\ &\leq 2 \underbrace{\E\big[ (m(X) - m^*_{\hat{T}} (\bR_{k}, X) )^2\big]}_{\text{\quot{Bias term}} } + 2\underbrace{\E\big[ (\hat{m}_{\hat{T}}(\bR_k,D_n,X) - m^*_{\hat{T}} (\bR_{k}, X) )^2\big] }_{\text{\quot{Estimation variance term} } }
\end{align*}
Following \citet{Chi}, we analyze both summands separately. This results into \Cref{lemma:error1,lemma:error2} in \Cref{subsec:notes_proof}. We note that the analysis for the estimation variance term follows readily from the approximation theory developed by \citet{Chi}, see also the notes after the statement of \Cref{lemma:error2} in \Cref{subsec:notes_proof}. Therefore, let us focus on how to deal with the bias term. As mentioned, the main difference is that we need to incorporate our randomized ($2$-step) SID Condition \ref{cond:C1}. First, it is a simple consequence of \ref{cond:C1} that the following variant holds, provided that $W$ is chosen appropriately.
\begin{condition}{(C1')}\label{cond:c1'}
    	There exist $\alpha_1 \geq 1$, $\delta \in [0.75,1]$, $W\in \N$ such that for any rectangular cell $t =\bigtimes_{j=1}^d t^{(j)}$ we have
    	\begin{align}\label{eq:C1'-good}
    		\var( m(X) | X\in t) \leq \alpha_1 \sup_{w,\bc_1,\bc_2} \zwei_t( \bc^w| \bc_1, \bc_2)   \text{ with probability at least } \delta.
    	\end{align}
    	Here, the supremum is over $w \in \{1,\dots, W\}$ and split points $\bc_1,\bc_2 \in \{1,\dots,d\} \times [0,1]$. Furthermore, $\bc^w = (\gamma^w,U^w)$, $w=1,\dots,W$ where $(\gamma^1,U^1),\dots,(\gamma^W, U^W)$ are $i.i.d.$ random variables with $\PP( \gamma^w = j ) = \frac{1}{d}$, $j \in \{1,\dots, d\}$ and, conditionally on $\gamma =j$, $U^w$ is uniform on $t^{(j)}$.
\end{condition}
The probability $\delta \geq 0.75$ may seem somewhat arbitrary but it is necessary for the following property to hold. Suppose that $T$ is a tree growing rule of depth $k$ associated to $\bR_k$. Then under \ref{cond:c1'}, as $k \to \infty$, the probability that the inequality \eqref{eq:C1'-good} holds true at least once in any tree branch, converges to $1$, $k\to \infty$. This allows us to ensure that the inequality from \ref{cond:c1'} holds true sufficiently often within tree branches. More precisely, we can deduce that, on an event with high probability, Condition \ref{cond:C6'} holds.
\begin{condition}{(C6')}\label{cond:C6'}
    A tree growing rule $T$ with root $\rootcell$, associated to a split determining sequence fulfills condition (C6') with $\alpha_1\geq 1 $, $M \in \N$, if for any $a\in\{1,2\}^k$ (i.e. for any tree branch) there exist at least $M$ (C1')-good cells among $\rootcell_{a_1\dots a_{l-2}}$, where $l \in \{2,4,\dots, k\}$. By a (C1')-good cell we denote a cell $t$ such that the inequality         \begin{align*}
            \var( m(X) | X\in t) \leq \alpha_1 \sup_{w,\bc_1,\bc_2} \zwei_t\left( \bc^w |\bc_1,\bc_2 \right)
        \end{align*}
    holds.
\end{condition}
To sum this up, thanks to our \Cref{cond:C1} and due to construction of the RSRF tree $\hat{T}$ it can be shown that, on some high probability event, $\hat{T}$ fulfills \Cref{cond:C6'}.\\ \ \\
Next, let us illustrate how the impurity gain inequality \eqref{eq:C1'-good} is used. Suppose $t$ is a (deterministic) cell such that the inequality $\var(m(X)|X\in t) \leq \alpha_1 \zwei(t;t_{\bullet \bullet})$ holds with daughter cells $t_{\pp}= (t_{a})_{a \in \{1,2\}}$. Then, 
\begin{align*}
    \eins( t; t_{\pp} ) = \var(m(X) |X\in t) - \zwei(t;t_{\pp})\leq \var(m(X) |X\in t) (1-\alpha_1^{-1} ).
\end{align*}
By the definition of $\eins$ and \Cref{remark:eins_plus_zwei}, it holds that
\begin{align*}
    \E\bigg[ \Big( m(X) - \sum_{ a \in \{1,2\}^2} \1_{(X\in t_a)}\E[m(X)|X\in t_{a}] \Big)^2 \bigg] &= \sum_{a\in \{1,2\}^2 } \PP(X \in t_{a}) \var(m(X)| X \in t_a) \\&\leq \PP(X\in t) \var(m(X) |X\in t) (1-\alpha_1^{-1}),
\end{align*}
and $(1-\alpha_1^{-1})$ is a constant strictly smaller than $1$. In our case, this argument can be repeatedly applied $M$ times, in view of \ref{cond:C6'}. However, the cells of $\hat{T}$ are not deterministic as they are subject to sample variation (and the randomization). Therefore, another intermediate condition, \ref{cond:C5'} below, is needed. This is a variant of Condition 5 of \citet{Chi} taylored to the present setup. At this point, two things are to be shown
\begin{itemize}
    \item If a tree growing rule $T$ satisfies Conditions \ref{cond:C5'} and \ref{cond:C6'}, then the error $\E[ (m(X) - m^*_{T}(X) )^2]$ can be controlled
    \item The tree growing rule $\hat{T}$ (more precisely, a variant of it) satisfies Conditions \ref{cond:C5'} and \ref{cond:C6'} with high probability, conditionally on data and randomization.
\end{itemize}
    \begin{condition}{(C5')}\label{cond:C5'}
    Let $T$ be a tree growing rule of depth $k$ associated with a split determining sequence. Denote the root cell by $\rootcell$. Given $a \in \{1,2\}^k$, define for $\varepsilon >0$ 
        \begin{align*}
        L_{\varepsilon, a} := \Big\lbrace l\in\{2,\dots, k\} \ : \text{for }t=\rootcell_{a_1\dots a_{l-2}} \text{ it holds that }\zwei(t,t_{\pp}) \leq \varepsilon \Big\rbrace.
        \end{align*}
    We say that $T$ fulfills condition (C5') with $\varepsilon > 0, \alpha_2 \geq 1$ if for all $a \in \{1,2\}^k$ (i.e. for all tree branches in $T$) it holds that
    	\begin{align*}
			l \in L_{\varepsilon, a} &\implies \text{For the cell $t=\rootcell_{a_1\dots a_{l-2}}$ it holds that }\sup_{w, \bc_1,\bc_2} \zwei_{t}(\bc^w | \bc_1,\bc_2) \leq \alpha_2 \varepsilon \\
			l \notin L_{\varepsilon, a} &\implies \text{For the cell $t=\rootcell_{a_1\dots a_{l-2}}$ it holds that } \sup_{w, \bc_1,\bc_2} \zwei_t \left( \bc^w| \bc_1,\bc_2 \right) \leq \alpha_2 \zwei( t; t_{\pp}).
		\end{align*}
    Note that above, for each cell $t$ in the tree, $\bc^w$ denotes the split points induced by the split determining sequence.
\end{condition}

 \subsection{Proof of \Cref{thm:main}}\label{subsec:notes_proof}
	\Cref{thm:main} is a consequence of \Cref{lemma:error1,lemma:error2} below analyzing the \quot{bias term} and \quot{estimation variance term} seperately. While \Cref{lemma:error2} follows from \citet{Chi}, essentially new arguments are necessary for the proof of \Cref{lemma:error1}, due to the random nature of our SID condition.
 \begin{proofname}{Proof of \Cref{thm:main}.} The proof combines both \Cref{lemma:error1,lemma:error2}. Note that, here, given some $\varepsilon>0$, we have to choose $M$ such that $(1-(\alpha_1\alpha_2)^{-1})^M < \varepsilon$. Then \Cref{lemma:error1} is used with this choice of $M$. The remaining details on the combination of \Cref{lemma:error1,lemma:error2} are along the lines of \citet[Proof of Theorem 1 in Appendix A.3]{Chi}.

    \end{proofname}
	\begin{lemma}[\quot{Bias term}]\label[lemma]{lemma:error1}
	Let $\hat{T}$ be the Random-Sample-CART tree growing rule of depth $k$ and assume Conditions \ref{cond:C1}, \ref{cond:C2}, \ref{cond:C3}, \ref{cond:C4} and let $k=k_n$ and $W$ be as in \Cref{thm:main}. Let $\alpha_2> 1$, $\eta \in (0, \frac{1}{8})$ and $\tau \in (2\eta, \frac{1}{4})$. Let $\varepsilon'> 0$ and $M \in \N$ be arbitrary. Then, for all large $n$,
		\begin{align*}
			\E\big[ (m(X) - m^*_{\hat{T}} (\bR_{k}, X) )^2\big] \leq 8 M_0^2 n^{-\tau}2^k + 2 \alpha_1\alpha_2 n^{-\eta} + 2 M_0^2 (1-(\alpha_1 \alpha_2 )^{-1})^M + 2n^{-1} + \varepsilon'.
		\end{align*}
  In particular, 
		\begin{align*}
			\lim_{n\to\infty} 	\E\big[ (m(X) - m^*_{\hat{T}} (\bR_{k}, X) )^2\big]  = 0.
		\end{align*}
	\end{lemma}
	\begin{lemma}[\quot{Estimation variance term}, see {\citeauthor{Chi}, \citeyear{Chi}}, Lemma 2.]\label[lemma]{lemma:error2}
		Assume that Conditions \ref{cond:C2}, \ref{cond:C3} and \ref{cond:C4} hold. Let $\eta \in (0,\frac{1}{4})$, $c\in (0, \frac{1}{4})$ and $\nu > 0$. Then there exists $C_1> 0$ such that for $n$ large enough and each sequence $k=k_n$ with $1 \leq k \leq c \log_2 (n)$
		\begin{align*}
			\E\big[ (\hat{m}_{\hat{T}}(\bR_k,D_n,X) - m^*_{\hat{T}} (\bR_{k}, X) )^2\big] \leq n^{-\eta} + C_1 2^k n^{-\frac{1}{2} + \nu}.
	 	\end{align*}
	\end{lemma}
	For the proof of \Cref{lemma:error2} we refer to \citet[Lemma 2]{Chi}. Note that, therein, it is not used that the tree is grown using the Sample-CART-criterion. Instead, it is only important that the space $[0,1]^d$ is partitioned iteratively by placing rectangular cuts for any cell in the partition, starting with $[0,1]^d$. This is repeated at most $c \log_2 (n)$ times. Hence, the proof of \citet[Lemma 2]{Chi} applies to \Cref{lemma:error2} above, when $\hat{T}$ with depth $k$ is the Random-Sample-CART tree. We note that, for the analogous statement to \Cref{lemma:error2} under the general setting from \Cref{thm:main3}, more care is needed. See \Cref{subsec:proof-main3}. \\ \ \\
    To summarize, we are left with the proof of \Cref{lemma:error1} which will be addressed in \Cref{sec:proofs}. Proving \Cref{lemma:error1} is the main challenge of this paper.
\section{Proofs}\label{sec:proofs}
\subsection{Proof of \Cref{lemma:error1}}\label{subsec:proof_lemma_error1}
	As argued in \Cref{subsec:notes_proof}, to establish the consistency result from \Cref{thm:main}, it only remains to prove \Cref{lemma:error1}. This follows from some intermediate results and we outline how these are combined to get a proof for \Cref{lemma:error1}. The proof of \Cref{lemma:error1} itself is to be found at the end of this section.\\
 First, for a deterministic tree growing rule, it is shown that when it satisfies Conditions \ref{cond:C5'} and \ref{cond:C6'}, it is possible to control the error $\E[ (m(X) - m^*_{\hat{T}}(X) )^2]$. This is \Cref{thm3} below. Compare also \citet[Theorem 3]{Chi}.

	\begin{theorem}\label{thm3}
		Suppose $T$ is a tree growing rule of depth $k$ which fulfills Conditions \ref{cond:C5'} with $\alpha_2 \geq 1$, $\varepsilon \geq 0$ as well as \Cref{cond:C6'} with $M\in\N$. Then,
		\begin{align*}
			\E\Big[ \big( m(X) - m^*_T(X) \big)^2 \Big] \leq \alpha_1 \alpha_2 \varepsilon + (1-(\alpha_1\alpha_2)^{-1})^{M} \var(m(X) ).
		\end{align*}
	\end{theorem}
Then, it is shown that conditionally on the data $D_n$ and the randomization variables $\bR_{k}$, \Cref{cond:C5'} is satisfied, for a variant of $\widehat{T}$. Furthermore, on some high probability event, \Cref{cond:C6'} holds, assuming our \Cref{cond:C1}. All statements related to our \Cref{cond:C1} and their proofs are collected in \Cref{subsec:C1-related}.\\
That is, by conditioning on data $D_n$ and randomization $\bR_{k}$ we are able to apply \Cref{thm3} on some events with high probability. This will eventually lead to \Cref{lemma:error1}. \\
The variant of $\widehat{T}$ which fulfills \Cref{cond:C5'} is defined below in \Cref{def:semisample}. In words, as soon as tree cells become small, the Semi-Sample RSRF tree replaces the (Sample) CART splits by theoretical CART splits, and evaluates all candidate splits using the score $\zwei$, instead of $\widehat{\zwei}$.
\begin{definition}[Semi-Sample RSRF Tree]\label[definition]{def:semisample}
    Let $\hat{T}$ be the tree from RSRF based on $D_n$ and associated to the random split determining sequence $\bR_k$. Let us write $\widehat{T}_{\zeta}$, $\zeta\in [0,1]$, for the so-called Semi-Sample RSRF tree growing rule, which (given a realization of $\bR_k$) modifies $\widehat{T}$ as follows. For each $a \in \{1,2\}^k$ with $\PP(X \in \rootcell_{a}) < \zeta$, choose $l_0 := \min \big\lbrace l\in \{2,4,\dots, k\} : \PP(X \in \rootcell_{a_1\dots a_{l-2}}) < \zeta \big\rbrace$. Then, the tree branch (corresponding to $a$) is trimmed until depth $l_0-2$. New cells are then grown by iterating the following two steps starting from cell $t= \rootcell_{a_1\dots a_{l_0-2}}$.
    \begin{enumerate}
        \item Determine $\width$ many pairs $\bc^w = (j^w,c^w)\in \{1,\dots,d\} \times t^{(j)}$ according to the split determining sequence.
     \item Choose
     \begin{align*}
         (w^*, \bc^*_{1}, \bc^*_{2} ) \in \underset{w, \bc_1,\bc_2}{\arg\sup}\  \zwei_{t}( \bc^w| \bc_1,\bc_2).
     \end{align*}
     Split $t$ into $t_1, t_2$ at $\bc^{w^*}$ and then split $t_1$ at $\bc^*_{1}$ into $t_{1,1},t_{1,2}$. Define the daughter cells of $t_2$ analogously.
    \end{enumerate}
\end{definition}
	\begin{theorem}\label{thm:C5-fulfilled}
		Assume Conditions \ref{cond:C2}, \ref{cond:C3}, \ref{cond:C4}, $\eta \in (0, \frac{1}{8})$, $\tau \in (2\eta, \frac{1}{4})$, $k=k_n= c\log_2(n)$ with $c> 0$ and let $\zeta = n^{-\tau}$. Then, $\hat{T}_{\zeta}$ fulfills \Cref{cond:C5'} with $\varepsilon = n^{-\eta}$ on an event $U_n$ with $\PP(U_n) = 1 - o(n^{-1})$.
	\end{theorem}
	\begin{proofname}{Proof of \Cref{lemma:error1}}\label{proof:lemma:error1} Let $\zeta = n^{-\tau}$. We can bound
		\begin{align}\label{eq:replace-sonnenbaum}
			\E\big[ (m(X) - m^*_{\hat{T}}& (\bR_{k}, X) )^2\big] \notag\\ &\leq 2 \E\big[ (m(X) - m^*_{\hat{T}_\zeta}(\bR_{k}, X) )^2\big] + 2 \E\big[ \big( m^*_{\hat{T}}(\bR_{k}, X)  - m^*_{\hat{T}_\zeta}(\bR_{k}, X)  \big)^2\big].
		\end{align}
  Since \Cref{cond:C1} holds with some $\alpha_1\geq 1$ and $\delta>0$, we obtain from \Cref{lemma:C1-C1'} that \Cref{cond:c1'} holds with $\alpha_1$, by the choice of $W \geq \frac{-2\log(2)}{\log (1-\delta )}$. We start with the first summand in \eqref{eq:replace-sonnenbaum}. Let $M\in \N$ and let $B_n$ be the event that the tree $\hat{T}_\zeta$ contains at least $M$ (C1')-good cells in any tree branch, that is, if $t$ is such a cell, we have
  \begin{align*}
    		\var( m(X) | X\in t) \leq \alpha_1 \sup_{w,\bc_1,\bc_2} \zwei_t( \bc^w| \bc_1,\bc_2) .  
  \end{align*}
  Note that $B_n$ is measurable with respect to the data variables $D_n$ and randomization $\bR_{k}$. 
    By an application of \Cref{lemma:C1impliesC6} (which is a consequence of \Cref{lemma:goodsplits_everybranch}), we know that the probability of having at least $M$ \ref{cond:c1'}-good cells in every tree branch is close to $1$ as long as the depth of the tree is chosen large enough. Thus, using that $k=k_n \to \infty$, we obtain that condition (C6') is fulfilled on some event $B_n$ with high probability, for large enough $n$.
    Note that if we condition on $D_n$ and $\bR_{k}$, the tree growing rule $\hat{T}_\zeta$ becomes deterministic. Furthermore, it satisfies \Cref{cond:C5'} with some constant $\alpha_2$ and \Cref{cond:C6'} on $U_n \cap B_n$, by the argument above and in view of \Cref{thm:C5-fulfilled}. Note that the event $U_n$ comes from the statement of \Cref{thm:C5-fulfilled}. Thus, \Cref{thm3} may be applied. Hence, for arbitrary $\varepsilon'$, and large enough $n$, we have
	\begin{align*}
	\E\big[ &( m(X) - m^*_{\hat{T}_\zeta}(\bR_{k}, X)  )^2 \big] \\
	&=\E\Big[	\E\big[ ( m(X) - m^*_{\hat{T}_\zeta}(\bR_{k}, X)  )^2 \big| D_n, \bR_{k} \big] \Big] \\
	&=	 \E\Big[ (\1_{U_n\cap B_n} + \1_{U_n^c \cup B_n^c} ) 	\E\big[ ( m(X) - m^*_{\hat{T}_\zeta}(\bR_{k}, X)  )^2 \big| D_n, \bR_{k} \big] \Big] \\
	&\leq \E\Big[ \1_{U_n\cap B_n}  \E\big[ ( m(X) - m^*_{\hat{T}_\zeta}(\bR_{k}, X)  )^2 \big| D_n, \bR_{k} \big] \Big] + Cn^{-1} + 4 M_0^2 \varepsilon' \\
	&\leq  \alpha_1 \alpha_2 n^{-\eta} + (1-(\alpha_1\alpha_2)^{-1})^{M} \var(m(X) ) + n^{-1} + 4M_0^2 \varepsilon' &&(\text{by \Cref{thm3}}).
	\end{align*}
	Here, we also used the boundedness assumption from \Cref{cond:C4}. Next, we deal with the second summand in \eqref{eq:replace-sonnenbaum}. Analogously to equation (A.83) of \citet[appendix]{Chi} and by using \Cref{cond:C4} again, it holds that
	\begin{align*}
		\E\big[ (m^*_{\hat{T}}(\bR_{k}, X) - m^*_{\hat{T}_{\zeta}} (\bR_{k}, X)^2) \Big] \leq \zeta 2^{k+2} M_0^2.
	\end{align*}
	The second statement follows by noticing that we assumed $c < \frac{1}{4}$ in the definition of $k$, and by choosing $\tau$ and $\eta$ accordingly. This finishes the proof.
	\end{proofname}
\subsection{Lemmas Related to Condition \ref{cond:C1}}\label{subsec:C1-related}

    \begin{lemma}\label[lemma]{lemma:C1-C1'}
        Assume \Cref{cond:C1} holds with some $\alpha_1 \geq 1$, $\delta > 0$. Furthermore, suppose $W \in \N$ with $W \geq \frac{-2\log(2)}{\log (1-\delta )}$. Then, \ref{cond:c1'} holds with $\alpha_1$, $W$.
    \end{lemma}
    \begin{proof}
        Let $S_w : = \sup_{\bc_1,\bc_2}\zwei_{t}( \bc^w| \bc_1,\bc_2)$. Then, by \Cref{cond:C1}, there exists some $\alpha_1 \geq 1$ and $\delta > 0$ such that
        \begin{align*}
            \PP\left( S_w \geq \alpha_1^{-1} \var( m(X) |X \in t) \right) \geq \delta.
        \end{align*}
        Thus,
        \begin{align*}
            \PP\left( \max_{w=1,\dots,W} S_w < \alpha_1^{-1}\var( m(X)|X\in t) \right) &= \prod_{w=1}^W \left( 1- \PP\left( S_w \geq \alpha_1^{-1}\var( m(X)|X\in t) \right) \right) \\
            &\leq (1-\delta)^W
        \end{align*}
        Therefore and by the choice of $W$,
        \begin{align*}
            \PP\left( \max_{w=1,\dots,W} S_w \geq \alpha_1^{-1} \var(m(X)|X\in t) \right) \geq 1-(1-\delta)^W \geq 0.75.
        \end{align*}
    \end{proof}
    In the formulation of the lemma below, a cell $t\subseteq [0,1]^d$ is called (C1')-good if the inequality in \eqref{eq:C1'-good} is fulfilled.

\begin{lemma}[Good split in every branch]\label[lemma]{lemma:goodsplits_everybranch}
            Let $T$ be a tree growing rule (based on $D_n$) with associated randomization $\bR_k$ (see \Cref{subsec:treegrowingrules} for the definition of $\bR_k$) and assume the tree is grown up to depth $k$. Assume that \Cref{cond:c1'} holds with $\delta \geq 0.75$. Let $B(k)$ be the event that in every tree branch there exists at least one (C1')-good cell. Then,
            \begin{align*}
               1 - \PP( B(k) ) \leq \frac{2}{k}.
            \end{align*}
        \end{lemma}
        \begin{proof}
        Throughout the proof we assume that at level $0$ (root node) the first random split occurs, the next one is then at level $2$, and so on. Given a tree $T$ of depth $k$ and $a \in \{1,2\}^l$ for some $0 \leq l \leq k$ denote by $T^a$ the subtree starting at node $a$. We say that a branch in the tree $T^a$ is good if there exists some node in the branch which is (C1')-good. During this proof, we denote by $t = [0,1]^d$ the root node of $T$. Furthermore, we have randomization variables associated with the nodes in every second level of the tree, starting at the root node. We need some notation for the randomization up to a specific tree level: For even $l$, denote by $\bR_{0:l-2}$ the random variables associated with tree levels $0,2,\dots, l-2$.\\
        We deduce a recursive formula which will help us to prove the result. Let us fix some even $0 \leq l  \leq k-2$ and $a \in \{1,2\}^l$. Let $B_{a}$ be the event that all branches within $T^a$ are good. Below, when $l=0$ and thus $a = \emptyset$, we use the notational convention that $\PP(B_a|D_n, \bR_{0,-2} ):= \PP(B_a|D_n) = \PP(B(k)|D_n)$.\\
        \begin{align*}
            \PP&\big( B_a |D_n, \bR_{0:l-2} \big) \\
            &= \PP\big(
             B_a, \text{cell $t_a$ not (C1')-good} |D_n, \bR_{0:l-2}  \big) + \PP( B_a, \text{cell $t_a$ (C1')-good} |D_n, \bR_{0:l-2} \big) \\
             &= \PP( \text{cell $t_a$ not (C1')-good} |D_n, \bR_{0:l-2} )\PP\big( \bigcap_{\substack{i=1,2\\j=1,2} }  B_{(a_1\dots a_l, i,j)}|D_n, \bR_{0:l-2}\big)  \\
             &\qquad \qquad+ \PP( \text{cell $t_a$ (C1')-good} |D_n, \bR_{0:l-2} \big)       \\
             &= \PP( \text{cell $t_a$ (C1')-good} |D_n, \bR_{0:l-2} \big)  \\
             &\qquad + \Big(1- \PP( \text{cell $t_a$ (C1')-good} |D_n, \bR_{0:l-2} \big) \Big) \E\Big[ \prod_{\substack{i=1,2 \\ j=1,2} } \PP\big(   B_{(a_1\dots a_l, i,j)} |D_n, \bR_{0:l} \big)       \big| D_n,\bR_{0:l-2} \Big]
        \end{align*}
            Note that we used the tower property and the fact that $B_{(a_1,\dots,a_l,i,j)}$ are conditionally independent given $D_n$ and $\bR_{0:l}$. 
            Now define
            \begin{align*}
               x^*_l := \inf_{a \in \{1,2\}^{l} }\PP( B_a | D_n, \bR_{0:l-2} ),
            \end{align*}
            and $x_0^* := \PP( B(k) | D_n)$.\\
            Next, we deduce the recursive inequality. Note that, by assumption,
            \begin{equation*}
                \PP( \text{cell $t_a$ (C1')-good} |D_n, \bR_{0:l-2} \big) \geq \delta.
            \end{equation*}
            Using this fact, we have
            \begin{align*}
                \PP(B_a|D_n, \bR_{0:l-2} ) &\geq \inf_{p \geq \delta} \left( p + (1-p) \E\Big[  \prod_{\substack{i=1,2\\j=1,2}} \PP( B_{a_1,\dots,a_l,i,j} | D_n, \bR_{0:l} ) \big| D_n, \bR_{0:l-2} \Big] \right)\\
                &= \delta + (1-\delta) \E\Big[ \prod_{\substack{i=1,2\\j=1,2}} \PP( B_{a_1,\dots,a_l,i,j} | D_n, \bR_{0:l} ) \big| D_n, \bR_{0:l-2} \Big] \\
                &\geq \delta + (1-\delta) \E\Big[ {x_{l+2}^*}^4 \big| D_n, \bR_{0:l-2} \Big]
            \end{align*}
            Thus, for any $l = 0, 2, \dots, k-2$,
            \begin{align*}
                x_l^* \geq \delta + (1-\delta) \E[ {x_{l+2}^*}^4 | D_n ,\bR_{0:l-2}]
            \end{align*}
            Setting $x_l := \E[x_{k-l}^*]$ for $l = 2,\dots, k$ we get the recursive inequality
            \begin{align*}
                x_l &\geq  \delta + (1-\delta)x^4_{l-2}, l = 4,\dots, k,
            \end{align*}
            and $x_2 \geq \delta$ (by our assumption). Furthermore, $x_k= \PP(B(k))$, our quantity of interest. We now prove by induction that $x_l \geq 1-\frac{2}{l}$ for $l=2,\dots,k$ which then implies the result in view of $x_k=\PP(B(k))$. For $l=2$, the claim is trivial. Assume the inequality holds for some $l$. Then, using $\delta \geq 0.75$, the monotonicity of $x\mapsto 0.75(1+\frac{x^4}{3})$ and the induction hypothesis, we obtain
	\begin{align}\label{eq:recursion}
			x_{l+2} \geq \frac{3}{4}(1+ \frac{1}{3} x_l^4) \geq \frac{3}{4}(1+ \frac{1}{3}(1-\frac{2}{l})^4) &= \frac{1}{4}\big( 3+ (1-\frac{2}{l})^4 )\\
   &\notag= \frac{1}{4}\big( 4 -\frac{8}{l} + \frac{24}{l^2} - \frac{32}{l^3} + \frac{16}{l^4} \big) \\
   &\notag= 1 - \frac{2}{l} + \frac{6}{l^2} - \frac{8}{l^3} + \frac{4}{l^4} \\
   &\notag=: 1- A_l \\
   &\notag\geq 1- \frac{2}{l+2},
		\end{align}
		since it is easily checked that $A_l \leq \frac{2}{l+2}$ for all $l\geq 1$. Thus, the inequality holds for $l+2$.
            
        \end{proof}

        	\noindent The next lemma states that \Cref{cond:C6'} holds with high probability under \Cref{cond:c1'}. It is based on \Cref{lemma:goodsplits_everybranch} above.
	\begin{lemma}\label[lemma]{lemma:C1impliesC6}
                 Let $T$ be a tree growing rule (based on $D_n$) with associated randomization $\bR_k$ (see \Cref{subsec:treegrowingrules} for the definition of $\bR_k$) and assume the tree is grown up to depth $k$. Assume that \Cref{cond:c1'} holds and let $M \in \N$. Then, there exists an event $A$ with $\PP(A) \to 1$, as $k\to \infty$, such that on $A$, any tree branch contains at least $M$ (C1')-good cells.
        \end{lemma}
	\begin{proof}
Let $k_3 \in \N$ and $\varepsilon > 0$ arbitrary. Let $k_1, k_2$ be such that $k_1+k_2 = k_3$. Denote the root cell of $T$ by $t$ in this proof. Let $A_{l_1, l_2}$ be the collection of subtrees of $T$ starting from the cell $t_{a}$ for some $a \in \{1,2\}^{l_{1}}$ and being of depth $l_2$ (for example, $A_{0,k}$ contains solely the original tree $T$, $A_{2, k-1}$ contains the two trees starting at $t_1$ and $t_2$ with with same leafs as $T$, etc.). Then
\begin{align*}
    (*) :&= \PP( \text{at least $2$ (C1')-good cells in any tree branch of $T$} ) \\ &\geq \PP( \text{the tree in $A_{0,k_1}$ contains at least one (C1')-good cell,} \\ &\qquad \qquad \text{any tree in $A_{k_1, k_2}$ contains at least one (C1')-good cell}) \\ &\geq \left( 1- \frac{2}{k_1} \right) \left( 1- \frac{2}{k_2} \right)^{2^{k_1}}
\end{align*}
where we used \Cref{lemma:goodsplits_everybranch} and the fact that the decisions whether a cell $t$ is (C1')-good or not are made independently. We can choose $k_1$ and $k_2$ (depending on $k_3$) such that $k_2/2^{k_1}$ and $k_1$ tend towards infinity for $k_3\to\infty$ and thus, $k_3$ can be chosen large enough such that $(*) \geq 1-\varepsilon$. The statement from the Lemma can be deduced by repeating this argument.
	\end{proof}
 \subsection{Proofs of \Cref{thm:C5-fulfilled,thm3} }\label{subsec:proof_C5-fulfilled_thm3}
  \begin{proofname}{Proof of \Cref{thm3}}
    The proof of this theorem makes use of the basic arguments in the proof of Theorem 3 of \citet{Chi} but the random nature of our SID condition needs some additional steps and arguments. Recall that $\rootcell$ denotes the root cell and the tree is grown until depth $k$. It holds that
    		\begin{align*}
    			\E \big(m(X) - m^*_T(X) \big)^2 
       = \sum_{ a \in \{1,2\}^k } \PP(X \in \rootcell_{a})\var( m(X) | X\in \rootcell_{a} ).
    		\end{align*}
    	Let us fix some $a \in \{1,2\}^k$ and thus some tree branch $ (\rootcell, \rootcell_{a_1}, \dots \rootcell_{a_1\dots a_k} )$ of depth $k$. By \ref{cond:C6'}, for each such $a$, there exist $M$ distinct indices $g_1(a) < \dots < g_M(a) \in \{0,2,\dots, k-2\}$ such that the inequality in \Cref{cond:c1'} holds true for the cell at level $g_j(a)$, $j=1,\dots,M$. Within this proof, we call the indices $g_j(a)$, $j=1,\dots,M$ $(C1')$-good. Note that, without loss of generality, it can be assumed that if an index $g$ is among the $(C1')$-good indices for some $a \in \{1,2\}^k$, then this $g$ is also among the $(C1')$-good indices for all tuples $\tilde{a}$ whose first $g$ entries coincide with $a$. Specifically, if $g_M(a) = k-2$, then $g_M(\tilde{a}) = k-2$, too, for the four tuples $\tilde{a} \in\{1,2\}^k$ whose first $k-2$ entries coincide with $a$ (in other words, the corresponding tree branches coincide with each other up to depth $k-2$). Let us distinguish different cases. \\
    		\textbf{Case 1: It holds that $g_M(a)=k-2$ and for the cell $t=\rootcell_{a_1,\dots,a_{k-2}}$ it holds that}
      \begin{equation*}
          \sup_{w, \bc_1,\bc_2}\zwei_{t}( \bc^w | \bc_1, \bc_2) > \alpha_2 \varepsilon.
      \end{equation*}
      Then, by \Cref{cond:C5'}, it holds that $k \notin L_{\varepsilon, a}$. Again, by \Cref{cond:C5'}, this implies
    		\begin{align*}
    		 \sup_{w,\bc_1,\bc_2}\zwei_t( \bc^w| \bc_1, \bc_2) \leq \alpha_2 \zwei(t; t_{\pp} ).
    		\end{align*}
    		Thus, by $k-2$ being a $(C1')-$good index,
    		\begin{align*}
    		\begin{split}
    			\eins( t; t_{\pp} ) &= \var(m(X) |X\in t) - \zwei(t;t_{\pp}) \\
    			&\leq \var(m(X) | X\in t) - \alpha_2^{-1} \sup_{w, \bc_1,\bc_2}\zwei_t( \bc^w | \bc_1,\bc_2 ) \\
    			&\leq \var(m(X) |X\in t) (1-(\alpha_1\alpha_2)^{-1} ).
    		\end{split}
    		\end{align*}
    		Then, using the formula from \Cref{remark:eins_plus_zwei},\begin{align}\label{eq:improvementbound}
                 \begin{split}
    			\sum_{\substack{j=1,2\\l=1,2}} \PP(X\in t_{j,l})\var(m(X)|X\in t_{j,l}) &=  \PP(X\in t) \sum_{\substack{j=1,2\\l=1,2}}  \PP(X\in  t_{j,l} | X\in t ) \var( m(X) \in  t_{j,l}) \\&= \PP(X\in t) \eins(t; t_{\pp} ) \\& \leq \PP(X\in t) \var(m(X)| X\in t) (1-(\alpha_1\alpha_2)^{-1}).
                 \end{split}
    		\end{align}
    		\textbf{Case 2: $k$ is such that $g_M(a) = k-2$ and for $t=\rootcell_{a_1,\dots,a_{k-2}}$ it holds that}
                \begin{equation*}
                  \sup_{w,\bc_1,\bc_2}\zwei_t( \bc^w| \bc_1, \bc_2 ) \leq \alpha_2 \varepsilon.
              \end{equation*}
              Then,
    		\begin{align*}
    			\var(m(X)| X\in t) \leq \alpha_1 \alpha_2 \varepsilon
    		\end{align*}
    		since the cell at level $k-2$ is $(C1')$-good. Hence,
    		\begin{align}\label{eq:alphabound}
    		\begin{split}
    			\sum_{\substack{j=1,2\\l=1,2}} \PP(X\in t_{j,l})\var(m(X)|X\in t_{j,l})  &= \PP(X \in t)  \sum_{\substack{j=1,2\\l=1,2} } \PP(X \in t_{j,l} | X \in t )\var(m(X) | X \in t_{j,l})  \\&= \PP(X \in t) \eins( t; t_{\pp} )  \\&\leq \PP(X \in t) \alpha_1\alpha_2 \varepsilon. 
    		\end{split}
    		\end{align}
    		\textbf{In any case,} we always have the following simple bound which we will use from time to time. Let $t$ be some cell and $t_{j,l}$ its grand daughter cells. Then,
    		\begin{align}\label{eq:simplebound}
    			\begin{split}
    				\sum_{\substack{j=1,2\\l=1,2}} \PP(X \in t_{j,l} )\var(m(X)| X\in t_{j,l}) &= \PP(X\in t) \sum_{\substack{j=1,2\\l=1,2}} \PP(X \in t_{j,l} | X \in t )\var(m(X)| X\in t_{j,l}) \\
    				&=  \PP(X\in t) \eins( t; t_{\pp}  ) \\&\leq \PP(X\in t) \var(m(X)| X\in t),
    			\end{split}
    		\end{align}
     in view of \Cref{remark:eins_plus_zwei}. Recall that the tree $T$ is of depth $k$ and is identified with the set of $k-$tuples $\{1,2\}^k$. For $l\leq k$, let $T(l) := \{1,2\}^l$. The set $T(l)$ corresponds to the subtree of $T$ of depth $l$ with the same root $\rootcell$ as $T$. Furthermore, denote by $T_{\varepsilon}(l)$ the following set
        \begin{align*}
             a &\in T_{\varepsilon}(l) :\iff a \in \{1,2\}^l \text{ and }\\&\text{there exists an integer $g$ from the set of indices of all $(C1')$-good splits such that:} \\ 
       &\qquad \text{For the cell $t = \rootcell_{a_1\dots a_{g}}$ at depth $g$, it holds
                 } \sup_{w, \bc_1,\bc_2} \zwei_t( \bc^w| \bc_1, \bc_2) \leq \alpha_2 \varepsilon. 
        \end{align*}
         We distinguish in the following between tree branches corresponding to $T_\varepsilon(l)$ and those corresponding to $T^{\dagger}(l) := T(l)\setminus T_{\varepsilon}(l)$.\\
        \begin{align*}
    			\sum_{ a \in \{1,2\}^k }&\PP( X \in \rootcell_a ) \var(m(X) | X \in \rootcell_a ) \\
    			&= \Big( \sum_{\substack{ a  \in T^{\dagger}(k) \\ g_M(a) < k-2 }}+ \sum_{\substack{a \in T^{\dagger}(k) \\ g_M(a) = k-2 }}  \Big) \PP( X \in \rootcell_{a}) \var(m(X) | X \in \rootcell_{a} ) \\
      &\hspace{7cm} + \underbrace{\sum_{\substack{a \in T_{\varepsilon}(k) }} \PP( X \in \rootcell_{a}) \var(m(X) | X \in \rootcell_{a} )}_{=:A}\\
    			&\leq  \sum_{ \substack{a\in T^{\dagger}(k-2)  \\ M \text{ $(C1')$-good splits}\\ \text{splits in branch}\\\text{corresponding to $a$} } } \PP(X\in \rootcell_{a}) \var(m(X)| X\in \rootcell_{a}) \\
    			&\qquad + (1-(\alpha_1\alpha_2)^{-1})  \sum_{ \substack{a \in T^{\dagger}(k-2) \\ M-1 \text{ $(C1')$-good}\\ \text{splits in branch}\\\text{corresponding to $a$} }} \PP(X\in \rootcell_a) \var(m(X)| X\in \rootcell_a ) \\ 
    			&\qquad + A
        \end{align*}
    For the first summand, we employed \eqref{eq:simplebound}. For the second one, we employed \eqref{eq:improvementbound} from case 1. Note that we have used two facts: First, if $a \in T(k) \setminus T_{\varepsilon}(k)$, then all four tree branches which conincide with the branch corresponding to $a$ up to level $k-2$ are element of $T(k)\setminus T_{\varepsilon}(k)$. Furthermore, if $a \in T(k)\setminus T_{\varepsilon}(k)$, then for each $l \leq k$, $(a_1,\dots, a_l) \in T(l)\setminus T_{\varepsilon}(l)$. We shall treat te summand $A$ further below. \\
    To the first two summands in the last expression, we can iteratively apply the same argument as above by distinguishing again the cases, if in a tree branch there is a $(C1')$-good split at the current depth or not. As there are $M$ $(C1')$-good splits in any tree branch, we can iterate this until we have discovered all $(C1')$-good splits. More formally, we repeat the argument until the trees considered are of depth $m^*$ where $m^* := \min\{ g_1(a) : a \in \{1,2\}^k \}$. Then, we are left with a summand of the form
    \begin{align*}
        \sum_{a  \in T^{\dagger}(m^*)}(1-(\alpha_1\alpha_2)^{-1})^M \PP(X\in t_a)\var(m(X)| X\in t_a)
    \end{align*}
    which is bounded above by $\var(m(X) )(1-(\alpha_1\alpha_2)^{-1})^M$.\\
    It remains to bound $A$. The argument is similar to \cite[page 9 of the supplementary material]{Chi}. We give a proof for the sake of completeness. Fix some $a \in \{1,2\}^k$ with the property that $a \in T_{\varepsilon}(k)$. Then, we can choose a smallest $l\in \{ g_1(a),\dots, g_M(a) \}$ such that for $t = \rootcell_{a_1\dots a_l}$ it holds
    \begin{align*}
        \sup_{w,\bc_1,\bc_2} \zwei_t( \bc^w| \bc_1,\bc_2) \leq \alpha_2 \varepsilon.
    \end{align*}
    Let us denote by $S^*$ the set of tuples in $T_{\varepsilon}(k)$ such that the first $l$ cells are corresponding to $a^* := (a_1,\dots, a_{l})$. Then, as in \eqref{eq:alphabound}, and writing $a^*b$ for the concatenation $a^*b=(a^*_1,\dots,a^*_l, b_1,\dots,b_{k-l})$ of $a^*$ and $b$,
    \begin{align*}
        \sum_{a\in S^*} &\PP(X \in \rootcell_a ) \var( m(X) |X \in \rootcell_a) \\&= \sum_{b \in \{1,2\}^{k-l} } \PP(X \in \rootcell_{a^*b})\var(m(X)|X\in \rootcell_{a^*b}) \\
        &= \PP(X\in \rootcell_{a^*}) \sum_{b\in\{1,2\}^{k-l} }\PP(X \in \rootcell_{a^*b} |X \in \rootcell_{a^*} )\var(m(X) | X \in \rootcell_{a^*b}) \\
        &\leq \PP(X \in \rootcell_{a^*}) \var(m (X) |X \in \rootcell_{a^*} )&& (\text{by \Cref{remark:eins_plus_zwei}}) \\
        &\leq \PP(X \in \rootcell_{a^*}) \alpha_1\alpha_2 \varepsilon.
    \end{align*}
    In the last step, we used that the cell is $(C1')$-good, due to the choice of $l$. Now, observe that for the set of tuples in $T_{\varepsilon}(k) \setminus S^*$, the same argument can be made, yielding another set of the form $S^*$. Finally, we can sum over all such sets of the form $S^*$, noticing that these are pairwise disjoint, and have thus $A \leq \alpha_1\alpha_2\varepsilon$.
\end{proofname}
\begin{proofname}{Proof of \Cref{thm:C5-fulfilled}}
The proof is analogous to the proof of \citet[Theorem 4]{Chi}.\\

    \end{proofname}
\subsection{Details on the Proof of \Cref{thm:main3}}\label{subsec:proof-main3}
\Cref{thm:main3} can be proven using the same arguments as for the proof of \Cref{thm:main} and we omit the details. However, let us point out some differences that needs to be taken into account regarding the approximation theory of \citet{Chi}. Since the collection $\mathcal{C}$ in \Cref{thm:main3} contains not only rectangles, a closer look needs to be taken to the approximation theory developed by \citet[Section A.1]{Chi} and to the proof of the statement in \Cref{lemma:error2} \citep[compare][Lemma 2]{Chi}. In fact, this affects the use of the grid from \Cref{def:grid}. The grid $G_n$ is used to bound suprema over tree growing rules, by reducing them to maxima over tree growing rules. Given a set $t$, recall the definition of $t^{\#}$. Let $\mathcal{G}_{n,k}$ be all possible cells $t^{\#}$ where $t$ is any cell that can be built through $k$ iterations of the algorithm. By \Cref{cond:T_dim} assumed \Cref{thm:main3}, the cardinality of $\mathcal{G}_{n,k}$ is bounded by $(Bn^{\beta})^{k}$. This bound is different than the one in equation (A.8) of \citet[supplement, p.2]{Chi}, however the arguments used in equation (A.115) and the following lines are not affected and thus, the first statement of Lemma 7 of \citet{Chi} remains valid for $\mathcal{G}_{n,k}$ (instead of $G_{n,k}$ in \citealp[~p.2, supplement]{Chi}). The same is easily verified for the other two statements of Lemma 7 of \citet{Chi}. Note that we have the following proposition.
\begin{proposition}\label[proposition]{prop:approx_number_datapoints}
    Let $\epsilon_2 > 0$ and $\mathcal{H}_n$ be a collection as in \Cref{cond:T_boundary} satisfying the first property therein. Furthermore, assume that $X_1,\dots,X_n$ are i.i.d. satisfying \ref{cond:C2}. Let \begin{align*}
        \sA := \bigcup_{H \in \mathcal H_n} \big\lbrace \text{the number of indices $i$ with }X_i \in \bigcup_{\bq \in H} B_{\bq} \text{ is at least } (\log n)^{1 +\epsilon_2} \big\rbrace.
    \end{align*}
    Then, $\PP(\sA) \leq  \# \mathcal H_n n^{-\epsilon (\log n)^{1+\epsilon_2} }$, where $\epsilon>0$ comes from the  grid definition in \Cref{def:grid}.
\end{proposition}
\begin{proof}
    Let $\lambda$ denote the $d$-dimensional Lebesgue measure. Note that \[\sup_{H\in\mathcal H_n}\lambda( \bigcup_{\bq\in H}B_{\bq}) \leq C g_n^{-1}.\]Here, $C$ is the constant from \Cref{cond:T_boundary}. This statement follows readily with the same calculations as in \citet[equations (A.4)-(A.6)]{Chi}.
\end{proof}
By the conditions in \ref{cond:T_boundary}, such an $\mathcal{H}_n$ exists with $\# \mathcal H_n \leq n^{\beta'}$ where $\beta' > 0$. Observe, that for $t'\subseteq t$, $t' \Delta t'^{\#} = \bdiff( t,t') \cup (t \Delta t^{\#})$ and the union is disjoint. Furthermore, note that $([0,1]^d)^{\#} = [0,1]^d$. From these properties, it is easy to deduce that, uniformly over all $t$ that can be obtained through $k$ iterations of the algorithms, $\lambda (t \Delta t^{\#}) \leq C k g_n^{-1}$ where $\lambda$ denotes the Lebesgue measure, and $C$ is the constant from \Cref{cond:T_boundary}. This is the property required in equation (A.2) of \citet{Chi}. Additionally by \Cref{prop:approx_number_datapoints}, it can be checked that on event $\mathcal{A}^c$, uniformly over all $t$ obtained through $k$ iterations, the number of $x_i$'s in $t\Delta t^{\#}$ is bounded by $k (\log n )^{1 +\varepsilon_2}$. This is equation (A.7) of \citet[supplement, p.2]{Chi}. Furthermore, it holds that $n\PP(\sA^c) \to 0$ as $n\to \infty$.\\
It follows that the statement of \Cref{lemma:error2} and the approximation theory developed by \citet{Chi} remain true under the setup and the assumptions of \Cref{thm:main3}.\\
We note that the choice of $\delta \geq 1 - L^{-1}$ in our probabilistic SID is needed to prove the following statement: The probability $p_k$ that a depth $k$-tree (where in each step, a cell is split into $L$ cells) contains at least one good split in any branch of the tree, converges to $1$. To see this, suppose without loss of generality that $\delta = 1-L^{-1}$. As in the proof for \Cref{lemma:goodsplits_everybranch}, one can deduce the recursive formula $p_k = \delta + (1-\delta)p_{k-1}^L$ and $p_1=\delta$. Then, using the same calculations as in \eqref{eq:recursion} one can prove inductively that $p_k \geq 1-k^{-1}$ for $(L \geq 3)$ and $p_k \geq 1-2k^{-1}$ for $L=2$, in view of the following Lemma.
	\begin{lemma}
		 Suppose $L \in \N$, $L \geq 3$ and let $f,g:\R \to \R$, $f(x) := \frac{1}{x+1} - \frac{1}{L}\left( 1- \left(1-\frac{1}{x} \right)^L \right)$ and $g(x) := \frac{2}{x+1} - \frac{1}{2}\left( 1- \left(1-\frac{2}{x} \right)^2 \right)$. Then $f(x) \geq 0$, $g(x) \geq 0$, for all $x \geq 1$.
	\end{lemma}
	\begin{proof}
		Observe that the derivative $f'$ is given by
		\begin{align*}
			f'(x) = \frac{\left(1-\frac{1}{x}\right)^{L-1}}{x^2} - \frac{1}{(x+1)^2}.
		\end{align*}
		It holds that $f(1) = \frac{1}{2} - \frac{1}{L} > 0$, and $f(x) \to 0$, $x\to \infty$. The claim follows from $f'(x) \leq 0$ for all $x \geq 1$. Let $\alpha = L-1$. For $x \geq 1$,
		\begin{align*}
			f'(x) \leq 0 \iff \frac{(x-1)^{\alpha}(x+1)^2 }{x^{\alpha +2}} \leq 1.
		\end{align*}
		Observe that
		\begin{align*}
			 \frac{(x-1)^{\alpha}(x+1)^2 }{x^{\alpha +2}} &= \left(  1- \frac{1}{x} \right)^{\alpha -2}  \left[ \left( 1  - \frac{1}{x} \right) \left( 1 + \frac{1}{x} \right) \right]^2 \\&= 		 \left(  1- \frac{1}{x} \right)^{\alpha -2}  \left( 1-\frac{1}{x^2} \right)^2 < 1.
		\end{align*}
		The other inequality follows from simple calculations.
	\end{proof}
\subsection{Proof of \Cref{prop:example}}\label{sec:proof_example}
In this section, we prove the two statements of \Cref{prop:example}. Recall that $\zwei_{t}(j,c)$ denotes the ($1$-step) decrease in impurity when splitting a cell $t$ at coordinate $j$ and position $c$. We need the following lemma.
    	\begin{lemma}\label[lemma]{lemma:stuff}
		Let $t := [a_1,a_2] \times [b_1,b_2] \times [c_1,c_2] \subseteq [0,1]^3$, $m: [0,1]^3\to \R, m(x)=(x_1-0.5)(x_2-0.5)+x_3$ and suppose that $X$ is uniformly distributed on $[0,1]^3$. It holds that
		\begin{align}
			&\E\big[(X_1-0.5)(X_2-0.5) | X_1 \in [a_1,a_2], X_2 \in [b_1,b_2] \big] = \frac{1}{4}(a_2+a_1 - 1)(b_2+b_1-1)\label{eq:ex-expectation12},\\
			&\E\big[X_3 | X_3 \in [c_1,c_2] \big] = \frac{c_1+c_2}{2}, \label{eq:ex-expectation3} \\
			&\var(m(X) | X \in t) = \frac{H(a_1,a_2)H(b_1,b_2)}{9} - \frac{1}{16} (a_2+a_1-1)^2(b_2+b_1-1)^2 + \frac{(c_2-c_1)^2}{12}, \label{eq:ex-var}\\
			&\hspace*{3cm}\text{where }H(x,y )=(x-0.5)^2+(x-0.5)(y-0.5) + (y-0.5)^2,\notag\\
			&\zwei_t( 3,c ) = \frac{(c-c_1)(c_2-c)}{4}, \label{eq:ex-II-coord3} \\
			&\zwei_t( 1,c ) = \frac{(c-a_1)(a_2-c)}{16} (b_1+b_2-1)^2, \label{eq:ex-II-coord1}\\
			&\sup_c \zwei_t (1,c)  = \frac{(a_2 - a_1)^2}{64} (b_1 + b_2 -1)^2, \label{eq:ex-II-coord1-max} \\
			&\sup_c \zwei_t (2,c)  = \frac{(b_2 - b_1)^2}{64} (a_1 + a_2 -1 )^2, \label{eq:ex-II-coord2-max}\\
			&\sup_c \zwei_t(3,c)  = \frac{(c_2-c_1)^2}{16}. \label{eq:ex-II-coord3-max}
		\end{align}
        In \Cref{eq:ex-II-coord1-max,eq:ex-II-coord2-max,eq:ex-II-coord3-max} the supremum is attained for the midpoints of the cell boundary in coordinate $1$, $2$ and $3$, respectively.
	\end{lemma}
	
	\begin{proof}
    	\begin{enumerate}
        	\item Proof of \eqref{eq:ex-expectation12}.
        		\begin{align*}
        			\E\big[ &(X_1  - 0.5)(X_2 - 0.5) | X_1 \in [a_1,a_2], X_2 \in [b_1,b_2] \big]\\
        			&= \frac{1}{a_2- a_1} \int_{a_1}^{a_2} (x-0.5) \dx x \frac{1}{b_2-b_1} \int_{b_1}^{b_2} (x-0.5)\dx x.
         		\end{align*}
         		We have
         		\begin{align*}
         			\int_{a_1}^{a_2} (x-0.5)\dx x = \frac{1}{2}(x-0.5)^2 \Bigg|_{a_1}^{a_2} = \frac{(a_2-0.5)^2 -(a_1-0.5)^2}{2} = \frac{(a_1+a_2-1)(a_2-a_1)}{2}.		\end{align*}
        	Thus,
        	\begin{align*}
        		\E[(X_1 - 0.5) (X_2 - 0.5) | X_1 \in [a_1,a_2], X_2 \in [b_1,b_2]] = \frac{1}{4}(a_1+a_2-1)(b_1+b_2-1).
        	\end{align*}
        	\item Proof of \eqref{eq:ex-expectation3}. Conditionally on $X_3 \in [c_1,c_2]$, $X_3$ is uniformly distributed on $[c_1,c_2]$.
        	\item Proof of \eqref{eq:ex-var}. We first calculate
        	\begin{align*}
        		&\E[(X_1- 0.5)^2(X_2- 0.5)^2  | X_1 \in [a_1,a_2], X_2 \in[b_1,b_2] ]\\
        		&=\frac{1}{(b_2-b_1)(a_2-a_1) } \int_{a_1}^{a_2} (x-0.5)^2 \dx x \int_{b_1}^{b_2} (x-0.5)^2 \dx x.
        	\end{align*}
        	It holds that, using $y^3 - z ^3 =  (y-z)(y^2+yz+z^2)$,
        	\begin{align*}
        		\int_{a_1}^{a_2} (x-0.5)^2 \dx x  &= \frac{1}{3} (x-0.5)^3 \bigg|_{a_1}^{a_2} \\&= \frac{1}{3} \big( (a_2-0.5)^3 - (a_1 -0.5)^3 \big) \\
        		&= \frac{1}{3} (a_2-a_1) \big[ (a_2- 0.5)^2 + (a_2 -0.5)(a_1-0.5) +(a_1 -0 .5)^2  \big] .
        	\end{align*}
        	Thus,
        	\begin{align*}
        		\E[(X_1- 0.5)^2(X_2- 0.5)^2  | X_1 \in [a_1,a_2], X_2 \in[b_1,b_2] ] = \frac{H(a_1,a_2) H(b_1,b_2)}{9}.
        	\end{align*}
        	Consequently,
        	\begin{align*}
        		\var( (X_1 - 0.5) (X_2 -0.5) | &X \in t ) \\
        		&=\E[(X_1- 0.5)^2(X_2 - 0.5)^2  | X_1 \in [a_1,a_2], X_2 \in[b_1,b_2] ] \\
        		&\qquad - \big( \E[(X_1- 0.5)(X_2 - 0.5)  | X_1 \in [a_1,a_2], X_2 \in[b_1,b_2] ] \big)^2  \\
        		&=\frac{H(a_1,a_2) H(b_1,b_2)}{9} - \frac{1}{16}(a_2 + a_1  - 1)^2(b_2 + b_1 - 1)^2.
        	\end{align*}
        	The claim follows from
        	\begin{align*}
        		\var&( m(X) | X \in t) \\
          &= \var( (X_1 - 0.5) (X_2 -0.5) | X \in t ) + \var( X_3 | X \in t) \\
        		&=\var\big(X_1 -0.5)(X_2 -0.5) | X_1 \in [a_1,a_2] , X_2 \in [b_1,b_2] \big) + \var(X_3 | X_3\in [c_1,c_2]).
        	\end{align*}
         	\item Proof of \eqref{eq:ex-II-coord1} and \eqref{eq:ex-II-coord1-max}. Let $C \in [a_1,a_2]$ and $t_1 = \{x \in t : x_1 \leq C \}$ and $t_2 = t \setminus t_1$ be the other daughter cell. By \Cref{eq:ex-expectation12,eq:ex-expectation3}, it holds that
         	\begin{align}\label{eq:ex-proof01}
         		 \E(m(X)| X\in t_1) - &\E( m(X)| X\in t_2) \\ &= \frac{1}{4} \big( (a_1+C-1)(b_1+b_2-1) - \frac{1}{4}(a_2+C-1)(b_1+b_2-1) \big) \\ &= \frac{1}{4} (b_1+b_2-1) (a_1-a_2).
         	\end{align}Then, employing \eqref{eq:1step-alternative} and \eqref{eq:ex-proof01},
         	\begin{align*}
         		\zwei( t; t_1,t_2) &= \PP(X\in t_1| X\in t) \PP(X\in t_2 | X \in t) \big( \E(m(X)| X\in t_1) - \E( m(X)| X\in t_2)   \big)^2 \\
         		&=\frac{1}{16}\frac{C-a_1}{a_2-a_1}\frac{a_2-C}{a_2-a_1} (b_1+b_2-1)^2(a_2-a_1)^2 \\
         		&=\frac{(C-a_1)(a_2-C)}{16} (b_1+b_2-1)^2 \\
         		&\leq \frac{(a_2-a_1)^2}{64}(b_1+b_2-1)^2.
         	\end{align*}
         	In the last step we used the inequality of arithmetic-geometric mean, i.e. $xy \leq \frac{(x+y)^2}{4}$ for $x,y\geq 0$. Note that there is equality if and only if $C-a_1 = a_2-C$, i.e. the value $C$ that maximizes the expression $\zwei_t(1,C)$ is $C= \frac{a_2+a_1}{2}$. This shows the claim.
         	\item The proof of \eqref{eq:ex-II-coord2-max}, \eqref{eq:ex-II-coord3} and \eqref{eq:ex-II-coord3-max} is analogous. 
    	\end{enumerate}
	\end{proof}
    Having this at hand, we prove the first part of \Cref{prop:example}.
    \begin{proofname}{Proof of statement \ref{itemprop:exSID} from \Cref{prop:example}}
        Recall that $t=[0,1]^2 \times [c_1,c_2]$. By \eqref{eq:ex-II-coord1-max}, \eqref{eq:ex-II-coord2-max} and \eqref{eq:ex-II-coord3-max} it immediately follows that
        \begin{align*}
            \sup_{j\in\{1,2,3\}, c\in t^{(j)} }\zwei_t(j,c) = \frac{c_1+c_2}{2}.
        \end{align*}
        Furthermore, we get
        \begin{align*}
            \var(m(X) | X\in t) = \frac{1}{144} + \frac{(c_2-c_1)^2}{2}.
        \end{align*}
    \end{proofname}

\begin{lemma}\label[lemma]{lemma:2 step > 1 step}
        Let $t=\bigtimes_{j=1}^d[a_j,b_j]$ be a cell and $j_0\in\{1,\dots,d\}$, $\beta_0 \in[a_{j_0},b_{j_0}]$. Then,
        \[
            \sup_{\bc_1, \bc_2} \zwei_t( j_0,\beta_0 | \bc_1,\bc_2) \geq \zwei_t(j,\alpha)
        \]
        for any $j=1,\dots,d$ and $\alpha\in[a_j,b_j]$.
    \end{lemma}
    \begin{proof}
    Let $j \in \{1,\dots, d\}$ and $\alpha \in [a_j,b_j]$ be arbitrary. Let $\tilde t_1 = \{x \in t | x_j \leq \alpha \}$, $\tilde t_2 := t \setminus \tilde{t}$. Let $t_1 = \{ x \in t |x_{j_0} \leq \beta_0 \}$ and $t_2 = t\setminus t_1$. First, let us assume that $j \neq j_0$. Define
    \begin{align*}
       \bar{t}_{j,l} := t_j \cap \tilde t_{l}, \ j=1,2, \ l=1,2.
    \end{align*}
    Clearly, the partition given by $(\bar{t}_{j,l})_{j,l=1,2}$ can be obtained from splitting $t_1$ and $t_2$ at $j,\alpha$, but also by splitting $\tilde t_1$ and $\tilde t_2$ at $j_0,\beta_0$. Hence, in view of \Cref{def:impurity_decrease} and \Cref{remark:2step-1step}, it holds that
    \begin{align*}
        \sup_{\bc_1, \bc_2} \zwei_t( j_0,\beta_0 | \bc_1,\bc_2) \geq \zweiargneu{4}\left(t; (\bar{t}_{j,l})_{j=1,2;l=1,2} ) \right) \geq \zweiargneu{2}(t; \tilde t_1, \tilde t_2 ).
    \end{align*}
    Here, for clarification, we used $\zweiargneu{L}$ instead of $\zwei$ for the decrease associated to a partition of size $L$. Consider now the case $j=j_0$. If $\alpha = \beta_0$, then the result follows directly from \Cref{remark:2step-1step}. Next, assume $\alpha < \beta_0$. Let
    \begin{align*}
        t_{11}' &= \tilde t_1 \\
        t_{12}' &= \{x \in t | \alpha < x_j \leq \beta_0 \} \\
        t_{21}' &= \{x \in t | \beta_0 < x_j \leq \beta' \} \\
        t_{22}' &= \{ x \in t | x_j > \beta' \}.
    \end{align*}
    Here, $\beta' > \beta_0$ is arbitrary. Clearly $t_{jl}'$, $j=1,2$, $l=1,2$ is a partition of $t$ and can be obtained by first splitting $t$ into $t_1$ and $t_2$ (i.e. at $j_0, \beta_0$), and then splitting $t_1,t_2$ accordingly.  Therefore, it suffices to show that
    \[ \zwei^{(2)}(t ; \tilde t_{1}, \tilde t_2) \leq \zweiargneu{4}\left(t ; (t'_{jl})_{j=1,2,l=1,2} \right). \]
    Note that $\tilde t_2 = t \setminus \tilde t_1 = t_{21}' \cup t_{22}' \cup t_{12}'$ and
    \begin{align*}
        \zwei^{(2)}(t ; \tilde t_{1},\tilde t_2 ) = \PP(t_{11}'| t)  [\mu(t)-\mu(t_{11}')]^2 + \PP( \tilde t_2| t) [\mu(t)-\mu(\tilde t_2)]^2,
    \end{align*}
    writing $\PP(t'|t)= \PP(X\in t'|X\in t)$. Observe that
    \begin{align*}
        \PP(\tilde t_2 |t)[\mu(t)-\mu( \tilde t_2 )]^2 \leq \PP( t_{12}' |t ) [\mu(t) -\mu(t_{12}')]^2 + \PP( t_{21}' \cup t_{22}' | t ) [\mu(t) -\mu(t_{21}' \cup t_{22}')]^2,
    \end{align*}
    which follows as in the proof of \Cref{remark:2step-1step}, see equation \eqref{eq:inter} therein. Repeating this argument (for the second summand on the right hand side) yields
    \begin{align*}
         \zweiargneu{2}(t ; \tilde t_{1}, \tilde t_{2} ) \leq \sum_{\substack{j=1,2 \\ l=1,2}} \PP(t'_{jl} | t ) [\mu(t)-\mu(t'_{jl})]^2 = \zweiargneu{4}\left(t; (t'_{jl})_{j=1,2,l=1,2} \right).
    \end{align*}
    The case $j=j_0$ and $\alpha > \beta_0$ can be shown analogously.
    \end{proof}
    
Next, we prove the second statement of \Cref{prop:example}.
	\begin{proofname}{Proof of statement \ref{itemprop:exPSID} from \Cref{prop:example}}
		In order to simplify the proof below, we use the notation 
		\[
		    t=[0.5-l_1,0.5+l_2]\times [0.5-m_1,0.5+m_2]\times[c_1,c_2]
		\]
		with $l_1,l_2,m_1,m_2\in[-0.5,0.5]$ and $-l_1< l_2$, $-m_1< m_2$ for an arbitrary rectangular set $t$. Now, given some arbitrary cut $\bc = (j,c) \in \{1,\dots, d\} \times t^{(j)}$ in the $j$-th coordinate of $t$, define
		\begin{align*}
			S & := \sup_{\bc_1,\bc_2} \zwei( \bc| \bc_1,\bc_2 ).
		\end{align*}
  Throughout this proof, we denote by $t_1$ and $t_2$ the two daughter cells obtained from splitting $t$ at $\bc$. We wish to bound
		\begin{align}
		    \frac{\var(m(X)|X\in t)}{S}\label{eq:example - var}
		\end{align}
		by a constant. From \Cref{lemma:stuff} we have for cells $t=[a_1,a_2]\times [b_1,b_2]\times[c_1,c_2]$ that
		\[
		    \var(m(X)|X\in t)= \frac{H(a_1,a_2)H(b_1,b_2)}{9} - \frac{1}{16} (a_2+a_1-1)^2(b_2+b_1-1)^2 + \frac{(c_2-c_1)^2}{12} =: A+C,
		\]
		where $H(x,y )=(x-0.5)^2+(x-0.5)(y-0.5) + (y-0.5)^2$, $C=\frac{(c_2-c_1)^2}{12}$ and $A=\var(m(X)|X\in t)-C$. Now, if $A\leq C$ by \Cref{lemma:2 step > 1 step} and equation \eqref{eq:ex-II-coord3-max} we have
		\begin{align*}
		    \frac{\var(m(X)|X\in t)}{S} & \leq \frac{32C}{(c_1-c_2)^2}=\frac{8}{3}.
		\end{align*}
		Now, without loss of generality assume $l_2\geq l_1$ and $m_2\geq m_1$. Note that this implies $l_2, m_2 >  0$. Let $\tau<\frac{1}{2}$. For $A\geq C$ we distinguish between the following cases.

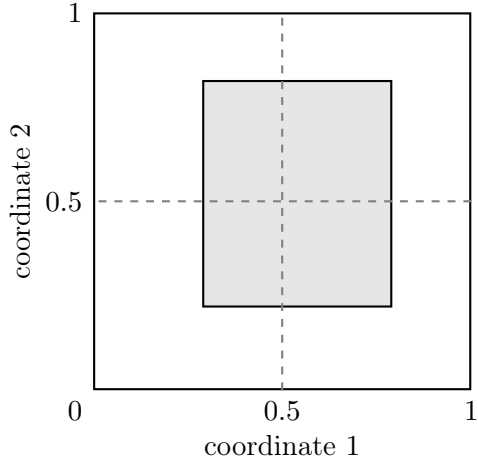
\begin{figure}[h]\centering
        \begin{tikzpicture}[thick,centered]
            \node[rectangle,draw, minimum height=5cm, minimum width=5cm] (r) at (0,0) {};
             \node[rectangle,draw, minimum height=3cm, minimum width=2.5cm, fill=gray!20] (s) at ($(r.center) + (0.2,0.1)$) {};

              \node[below, align = center] at (r.south) {$0.5$\\ \text{coordinate $1$} };
              \node[left] at (r.west) {$0.5$};
              \node[above, align=center, rotate = 90] at ($(r.west) - (0.7,0)$) {coordinate $2$};
              \node[below left] at (r.south west) {$0$};
              \node[left] at (r.north west) {$1$};
              \node[below] at (r.south east) {$1$};
              \draw[dashed, gray] (r.south) -- (r.north);
             \draw[dashed,gray ] (r.east) -- (r.west);

		\end{tikzpicture}
	\caption{In case \ref{case:middle} from the proof of the second statement of \Cref{prop:example}, the point $(0.5,0.5)$ is included in $[0.5-l_1,0.5+l_2]\times [0.5-m_1, 0.5+m_2]$}\label{fig:case-middle}
\end{figure}

		\begin{enumerate}
		    \item\label{case:middle} Let $\frac{l_1}{l_2}, \frac{m_1}{m_2}\geq 1-\tau$ (see \Cref{fig:case-middle}). Additionally, assume the first split is
      \begin{align}\label{eq:example - prob}
          \bc = (1, 0.5 +\alpha) \text{ or }\bc = (2,0.5+\alpha).
      \end{align}
		for some $\alpha\in(-l_1,l_2)$. We consider the case when $\bc = (1,0.5+\alpha)$. The other case is analogous. Then, following \Cref{remark:2step-1step} and \Cref{lemma:stuff}, we have
		    \begin{align*}
                  S & \geq \sup_{c_1,c_2} \left[ \PP(X \in t_1|X \in t)\zwei_{t_1}\left( 2,c_1 \right)+\PP(X \in t_2|X \in t)\zwei_{t_2}\left( 2,c_2 \right) \right]\\
		        & = \frac{\alpha+l_1}{l_1+l_2}\frac{1}{64}(m_1+m_2)^2(\alpha-l_1)^2+\frac{l_2-\alpha}{l_1+l_2} \frac{1}{64} (m_1+m_2)^2(\alpha+l_2)^2\\
		        & = \frac{(m_1+m_2)^2}{64(l_1+l_2)}\big((\alpha+l_1)(\alpha-l_1)^2+(l_2-\alpha)(\alpha+l_2)^2\big).
		    \end{align*}
		    By setting $\epsilon=l_2-l_1$ we obtain
		    \begin{align*}
		        (\alpha+l_1)(\alpha-l_1)^2 & = (\alpha+l_2-\epsilon)(\alpha-l_2+\epsilon)^2 \geq (\alpha+l_2)(\alpha-l_2)^2-4l_2^2\epsilon
		    \end{align*}
		    where the inequality can be checked by elementary calculations using that $l_1\leq l_2$. Furthermore, calculations yield
		    \begin{align*}
		        (\alpha+l_2)(l_2-\alpha)^2 + (l_2-\alpha)(\alpha + l_2)^2
    			=(\alpha +l_2 )(l_2-\alpha) 2l_2
    			&=2l_2(l_2^2-\alpha^2)\\&= 2l_2^3(1-\kappa^2)
		    \end{align*}
		    with $\kappa:=\frac{\alpha}{l_2}$. This implies 
		    \begin{align*}
		        S & \geq \frac{(m_1+m_2)^2}{128l_2}2l_2^3\left(1-\kappa^2-2\frac{\epsilon}{l_2}\right)=\frac{(m_1+m_2)^2}{64}l_2^2\left(2\frac{l_1}{l_2}-\kappa^2-1\right)\\
		        & \geq\frac{(m_1+m_2)^2}{64}l_2^2(1-2\tau-\kappa^2).
		    \end{align*}
		    Additionally, we have
		    \[
		        \var(m(X)|X\in t) \leq 2A \leq \frac{2(l_1^2-l_1l_2+l_2^2)(m_1^2-m_1m_2+m_2^2)}{9}
		    \]
                as well as
                \begin{align*}
                    \frac{m_1^2-m_1m_2+m_2^2}{(m_1+m_2)^2} = 1 - 3 \frac{m_1 m_2}{ (m_1 + m_2)^2 } \leq 1,
                \end{align*}
                because $m_1$ and $m_2$ have the same sign. This then yields
		    \begin{align*}
		        \frac{\var(m(X)|X\in t)}{S}&\leq \frac{128}{9}\cdot\frac{m_1^2-m_1m_2+m_2^2}{(m_1+m_2)^2}\cdot\frac{l_2^2}{l_2^2(1-2\tau-\kappa^2)}\\&\leq\frac{128}{9(1-2\tau-\kappa^2)}\\&\leq \frac{128}{9\gamma}
		    \end{align*}
		    for some $\gamma\in(0,1)$, if
		    \begin{align}
		        1-2\tau-\kappa^2\geq\gamma.\label{eq:example - inequ}
		    \end{align}
		    \item Let $\frac{l_1}{l_2}\leq 1-\tau$ and $\frac{m_1}{m_2}\geq \tau - 1$. 
                    We divide case (b) into the following cases and shall see that a common bound is
                    \begin{align*}
                        \frac{\var(m(X)|X\in t)}{S} \leq \frac{128}{9} \Big( 1 + 3 \frac{1-\tau}{\tau^2} \Big).
                    \end{align*}
                    \begin{itemize}
                        \item[(b.1)] Assume $\frac{m_1}{m_2} \leq 0$ and $\frac{l_1}{l_2} \geq 0$.
                        Following \Cref{lemma:2 step > 1 step} and \Cref{lemma:stuff} we obtain
                        \begin{align*}
                            S \geq \sup_{c}\ \zwei_t( 1,c ) = \frac{1}{64} (l_1+l_2)^2(m_2-m_1)^2.
                        \end{align*}
                        Using that $l_1$, $l_2$ (resp. $m_1, m_2$) have the same sign (resp. different signs) we can bound
                        \begin{align*}
                            \frac{l_1^2 - l_1l_2 + l_2^2}{(l_1+l_2)^2} = 1 -3  \frac{l_1l_2}{(l_1+l_2)^2} \leq 1
                            \end{align*}\text{and}\begin{align*}\frac{m_1^2-m_1m_2 + m_2^2}{(m_2-m_1)^2} = 1 + \frac{m_1m_2}{(m_2-m_1)^2} \leq 1.
                        \end{align*}
                        Using the bound for $\var(m(X)|X\in t)$ from (a) we get $\frac{\var(m(X)|X\in t) }{S} \leq \frac{128}{9}$.
                        \item[(b.2)] Assume $\frac{m_1}{m_2} \leq 0$ and $\frac{l_1}{l_2} \leq 0$. 
                         Again by \Cref{lemma:2 step > 1 step} and \Cref{lemma:stuff} it holds
                         \begin{align*}
                               S\geq \sup_{c}\ 
                               \zwei_t(2,c)
                               =\frac{(m_1+m_2)^2(l_2-l_1)^2}{64}.
                         \end{align*} 
                      Using the bound for $\var(m(X)|X\in t)$ from (a) and a bound as in case (b.1), we get
                      \begin{align*}
                             \frac{\var(m(X)|X\in t)}{S}\leq \frac{128}{9}\frac{(m_1^2-m_1m_2+m_2^2)}{(m_1+m_2)^2} &= \frac{128}{9}\Big( 1 - 3 \frac{m_1m_2}{(m_1+m_2)^2} \Big)\\
                             &= \frac{128}{9}\Big(1 - 3 \frac{m_1}{m_2(1+\frac{m_1}{m_2})^2} \Big).
                      \end{align*}
                       Observe that
                       \begin{align*}
                           \frac{\partial}{\partial x} \Big( 1 - \frac{3x}{(1+x)^2} \Big)= \frac{3(x-1)}{(1+x)^3} < 0,
                       \end{align*}
                        if $-1 <  x < 1$. Thus, we can bound
                        \begin{align*}
                             1 - 3 \frac{m_1}{m_2(1+\frac{m_1}{m_2})^2} \leq 1- \frac{3(\tau-1)}{\tau^2},
                        \end{align*}
                        which yields
                        \begin{align*}
                              \frac{\var(m(X)|X\in t)}{S}\leq \frac{128}{9} \Big( 1 + 3 \frac{1-\tau}{\tau^2} \Big).
                        \end{align*}
                        \item[(b.3)] Assume $\frac{m_1}{m_2} \geq 0$. Along the lines in (b.2), we have
                        \begin{align*}
                             \frac{\var(m(X)|X\in t)}{S}\leq \frac{128}{9}\frac{(l_1^2-l_1l_2+l_2^2)}{(l_2-l_1)^2} = \frac{128}{9} \Big( 1 + \frac{l_1}{l_2(1-\frac{l_1}{l_2})^2} \Big).
                        \end{align*}
                        Now, similarly to (b.2), by studying $f(x) = 1 + \frac{x}{(1-x)^2}$, we can bound
                        \begin{align*}
                             \frac{\var(m(X)|X\in t)}{S} \leq \frac{128}{9} \Big( 1 + \frac{1-\tau}{\tau^2} \Big).
                        \end{align*}
                    \end{itemize}
		    \item Let $\frac{l_1}{l_2}\geq \tau - 1$ and $\frac{m_1}{m_2}\leq 1-\tau$. This case is analogous to (b).
		    \item Let $\frac{l_1}{l_2},\frac{m_1}{m_2}\leq\tau-1$. Without loss of generality, assume $\frac{m_1}{m_2}\geq\frac{l_1}{l_2}$. Using Lemma \ref{lemma:2 step > 1 step} and Lemma \ref{lemma:stuff} we obtain
		    \[
		        S\geq \frac{(m_1+m_2)^2(l_2-l_1)^2}{64}
		    \]
		    as well as
		    \[
		        \var(m(X)|X\in t)\leq 2A.
		    \]
		    Thus 
		    \begin{align*}
		        \frac{\var(m(X)|X\in t)}{S} & \leq \frac{8}{9}\frac{16(m_1^2-m_1m_2+m_2^2)(l_1^2-l_1l_2+l_2^2)-9(m_2-m_1)^2(l_2-l_1)^2}{(m_1+m_2)^2(l_2-l_1)^2}\\  
		        & =: \frac{8}{9}F(m_1,m_2,l_1,l_2).
		    \end{align*}
		    Calculating the derivative with respect to $m_1$ we obtain
		    \begin{align*}
                \frac{\partial F}{\partial m_1} & =-\frac{12m_2(m_2-m_1)(l_1+l_2)^2}{(m_2+m_1)^3(l_2-l_1)^2}< 0,
		    \end{align*}
		    where we note that $\frac{l_1}{l_2},\frac{m_1}{m_2}\leq\tau-1$ together with $l_2\geq l_1$, $m_2\geq m_1$ imply that $l_2,m_2>0$, $l_1,m_1\leq 0$. Additionally, we used $-m_1<m_2$ by the definition of $t$. Now, since $\frac{m_1}{m_2}\geq \frac{l_1}{l_2}$ we have $m_1\geq \frac{m_2l_1}{l_2}$. Thus
		    \begin{align*}
		        \frac{\var(m(X)|X\in t)}{S} & \leq \frac{8}{9}\sup_{s\in [ \frac{m_2l_1}{l_2},m_2]}F(s,m_2,l_1,l_2)=\frac{8}{9}F\left(\frac{m_2l_1}{l_2},m_2,l_1,l_2\right)\\
		        & = \frac{8}{9}\frac{16(l_1^2-l_1l_2+l_2^2)^2-9(l_2-l_1)^4}{(l_2-l_1)^2(l_2+l_1)^2}\\
                    &  = \frac{8}{9} \frac{ \big( 4(l_1^2-l_1l_2+l_2^2) + 3(l_2-l_1)^2  \big) \big( 4(l_1^2-l_1l_2+l_2^2) - 3(l_2-l_1)^2 \big)  }{(l_2-l_1)^2(l_2+l_1)^2} \\
                    &   = \frac{8}{9} \frac{ 7l_1^2 - 10 l_1l_2 + 7l_2^2 }{(l_2-l_1)^2}  \\
		        & =: \frac{8}{9} G(l_1,l_2).
		    \end{align*}
		    Calculating the derivative with respect to $l_1$ yields
		    \[
		        \frac{\partial G}{\partial l_1}=\frac{4l_2(l_1+l_2)}{(l_2-l_1)^3}>0.
		    \]
		    Since $\frac{l_1}{l_2}\leq \tau-1$ we have $l_1\leq l_2(\tau-1)$. Thus
         \begin{align*}        
         \frac{\var(m(X)|X\in t)}{S}  \leq \frac{8}{9} \sup_{ s\in[-l_2, l_2(\tau-1)]} G(s,l_2) = \frac{8}{9}G(l_2(\tau-1),l_2) = \frac{8}{9}\frac{7\tau^2 - 24 \tau + 24}{(2-\tau)^2} 
      \end{align*}
  
		    Note that we required \eqref{eq:example - prob} as well as \eqref{eq:example - inequ} in order to bound \eqref{eq:example - var} in (a). Note that 
		    \begin{align*}
		        \eqref{eq:example - inequ} \Leftrightarrow \alpha^2\leq l_2^2(1-2\tau-\gamma)\Leftrightarrow -l_2\sqrt{1-2\tau-\gamma}\leq\alpha \leq l_2\sqrt{1-2\tau-\gamma}
		    \end{align*}
		    and thus
		    \begin{align*}
		        \PP(\eqref{eq:example - prob}, \eqref{eq:example - inequ})=\PP (\eqref{eq:example - inequ}|\eqref{eq:example - prob})\PP(\eqref{eq:example - prob})\geq \sqrt{1-2\tau-\gamma}\frac{2}{3}.
		    \end{align*}
		    This implies the assertion, since the right hand side converges to $\frac{2}{3}$ for $\tau,\gamma\rightarrow 0$.
		\end{enumerate}
	\end{proofname}

\subsection{Verification of Conditions \ref{cond:T_dim} and \ref{cond:T_boundary} for \Cref{ex:oblique,ex:interaction_forests}}\label{subsec:proof_oblique}
We check Conditions \ref{cond:T_dim} and \ref{cond:T_boundary} for
Oblique Trees (\Cref{ex:oblique}) and Interaction Forests (\Cref{ex:interaction_forests}). Recall the definition of the grid and related notation from \Cref{def:grid}.
     \begin{proofname}{Proof of Conditions \ref{cond:T_dim} and \ref{cond:T_boundary} for Example \ref{ex:oblique}}
            We start with proving \Cref{cond:T_dim}. Observe that it is sufficient to assume $d=2$ and to show that for $t=[0,1]^2$
            \begin{align}\label{eq:complexity_simple}
                \# \lbrace G_n^{\#} \cap t': t' \in \mathcal{C}_1(t) \rbrace \leq C g_n^{c_0},
            \end{align}
            where $g_n$ is the number of grid points in one dimension defined in \Cref{def:grid} and $C>0$ is a constant. Then, for general $d$, a bound in \eqref{cardinality-bound} is given by $Cd^2g_n^{c_0}$ and the claim follows due to the assumption that $d$ is of polynomial order in $n$, and $g_n = O(n^{1+\varepsilon})$. In order to derive \eqref{eq:complexity_simple}, a simple bound is given by the fact that the grid points within the unit square $[0,1]^2$ can be separated into at most $2{(g_n+1)^2  \choose 2} = O(g_n^4)$ subsets (this follows from taking all ${(g_n+1)^2 \choose 2}$ pairs of grid points and drawing the line crossing through them).This establishes \Cref{cond:T_dim}\\
	Next, we check \Cref{cond:T_boundary}. Let $\mathcal{K}$ be the sets of the form $K = \{ x\in [0,1]^d : b_1x_{k_1} + b_2 x_{k_2} \leq s\}$ or $t = \{ x\in [0,1]^d : b_1x_{k_1} + b_2 x_{k_2} \geq  s\}$, where $b_1,b_2 \neq 0$, $s \in \R$ and $k_1,k_2 \in \{1,\dots, d\}$. By the definition of the algorithm, it easily checked that \ref{cond:T_boundary} follows if there exist constants $C,C',\beta >0$ and a set $\mathcal{H}$ such that (I) for each $H\in \mathcal H$, $\# H \leq C g_n^{d-1}$, (II) for each $K\in\mathcal{K}$ there exists $H\in\mathcal{H}$ such that $K \Delta K^{\#} \subseteq \bigcup_{q\in H} B_q,$ and (III) $\# \mathcal H \leq C'n^{\beta}$.
Let us describe the construction of such a set $\mathcal{H}$. Suppose w.l.o.g. that $K =  \{ x\in [0,1]^d : b_1x_{1} + b_2 x_{2} \leq s\}$ and let $x \in K \Delta K^{\#}$. Denote by $B(x)$ the box $B \in G_n$ such that $x \in B$. Let $\partial K$ be the boundary of $K$. Let $x'\in \partial K \cap B(x)$. Let $b = (b_1,b_2,0,\dots, 0)^t \in \R^d$. Let $d(x,\partial K)$ be the minimal distance between $x$ and boundary of $K$. Then by definition of $b$, elementary geometry and the Cauchy-Schwartz inequality,
\begin{align*}
	d(x,\partial K) = \frac{| b_1 x_1 + b_2 x_2 - s |}{\| b\|} &= \frac{\left| \langle b, x-x' \rangle  + \langle b, x' \rangle -s  \right|}{\|b\|} = \frac{ \left| \langle b, x-x' \rangle \right|}{\|b\|} \\&\leq \sqrt{ (x_1- x_1')^2 + (x_2 - x_2 ')^2 }\leq \sqrt{2}g_n,
\end{align*}
since $x' \in K$ and both $x,x' \in B(x)$. Consequently, $K \Delta K^{\#} \subseteq \{  x \in [0,1]^d : s - \sqrt{2} g_n^{-1}\|b\| \leq  b_1 x_1 + b_2 x_2 \leq s+\sqrt{2}g_n^{-1}\|b\| \} =: \bar{K}$. 

Now since any such $K$ can be represented by a line connecting any two points $y,z$ on the boundary of $[0,1]^2$, we can w.l.o.g. assume $d=2$. Furthermore, we can assume w.l.o.g. that $y_2=1$ and $z_2=0$, which means that the angle between the line and $\{x_2=0\}$ resp. $\{x_2=1\}$ is at most $45$ degrees. Let $\underline{y_1}$, $\underline{z_1} \in \{-\frac{1}{g_n},0,\frac{1}{g_n},\dots, 1, 1 + \frac{1}{g_n}\}$ resp. $\overline{y_1},\overline{z_1}$ be the closest grid points smaller or equal (resp. larger) than $y_1$ and $z_1$ and write $\underline{y}= (\underline{y_1}, 1)^t, \underline{z}=(\underline{z_1},0)^t$ and $\overline{y},\overline{z}$ are defined similarly. Furthermore, define $\underline{y^*}$, $\overline{y^*}$,  $\underline{z^*}$ $\overline{z^*}$, where this time we go two grid positions to the left resp. right from $\underline{y}$, $\underline{z}$, resp. $\overline{y}$, $\overline{z}$. Here, we extend the grid positions to the left and right of the interval $[0,1]$. Let $G^{\text{lower}}$ be the straight line connecting $\underline{y^*}$, $\underline{z^*}$ and let $G^{\text{upper}}$ be the one connecting $\overline{y^*}$ with $\overline{z^*}$. See \Cref{fig:proof_oblique} for an illustration. Now observe the following.
\begin{itemize}
	\item The distance between $y$ and the points where $\partial \bar K$ intersects with $\{ x_2=1  \}$ is at most $2 g_n^{-1}$ (recall that $g_n$ is the grid parameter). This follows from simple geometric reasoning using that $\sin(\pi/4) = 1/\sqrt{2}$. The same holds for $z$ and $\{x_2=0\}$. Thus, $\bar{K}$ is contained in the area between $G^{\text{lower}}$ and $G_{\text{upper} }$,
	\item the area between $G^{\text{lower} }$ and $G^{\text{upper} }$ is $5g_n^{-1}$,
	\item each of $G^{\text{lower}}$ and $G^{\text{upper} }$ can cut at most $2g_n$ boxes in the grid. 
\end{itemize}
Thus, the union of these boxes and the area between $G^{\text{lower}}$ and $G^{\text{upper}}$ cover $\bar{K}$, and it can contain at most $9g_n$ boxes $B \in G_n$. This shows (I), (II). Furthermore, from the construction, the number of different such sets needed to cover any possible $\bar{K}$ is of order $d^2g_n^2$, establishing (III).
\begin{figure}
    \centering
    \includegraphics[scale=0.7]{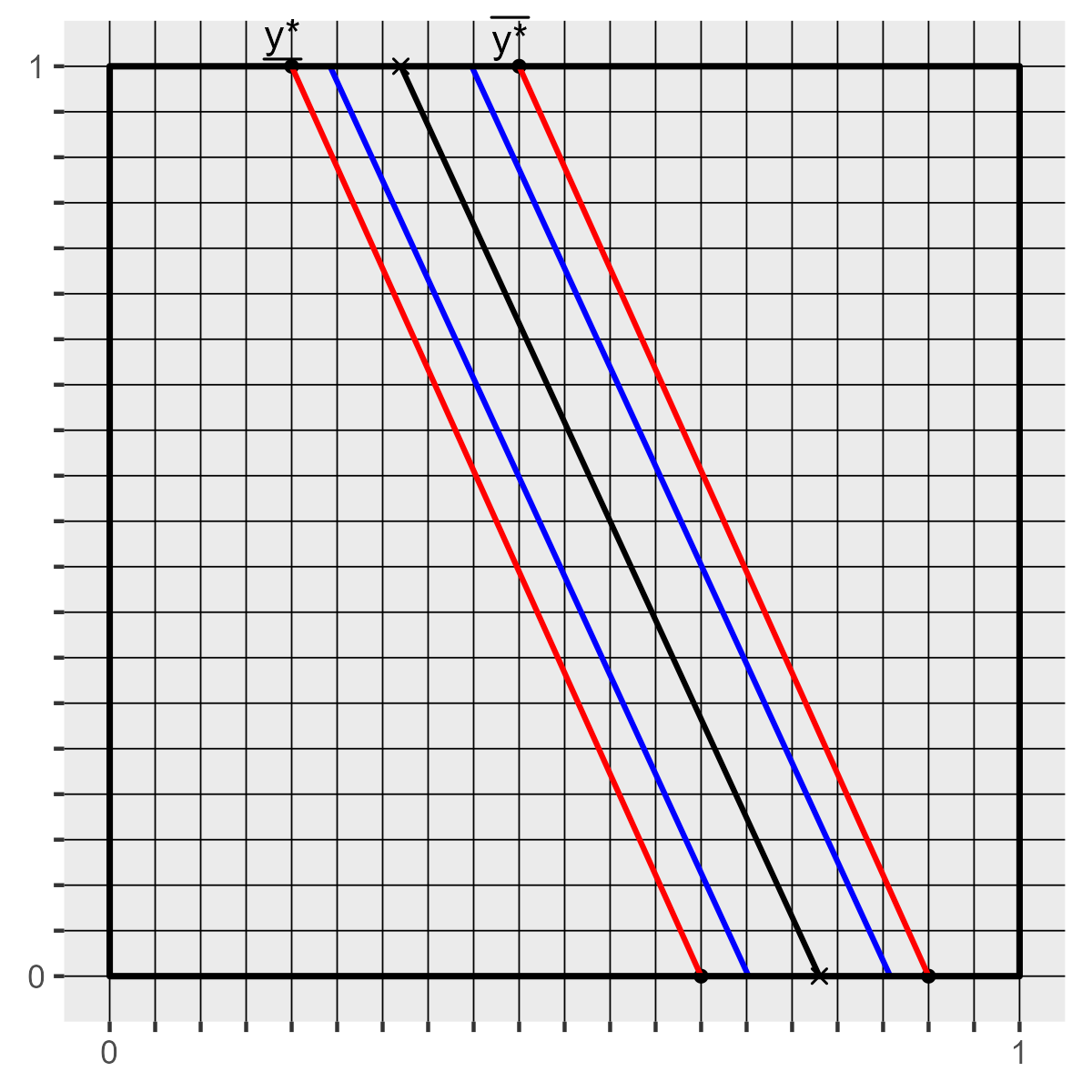}
    \caption{Illustration to the proof of \Cref{ex:oblique}. The red lines are $G^{\text{lower}}$ and $G^{\text{upper}}$. The set $\bar{K}$ is contained between the blue lines.}
    \label{fig:proof_oblique}
\end{figure}
\end{proofname}
\label{subsec:proof-main2}
\begin{proofname}{Proof of Conditions \ref{cond:T_dim} and \ref{cond:T_boundary} for \Cref{ex:interaction_forests}} We make use of \Cref{thm:main3}.
Let $t$ be a subset that can be constructed through the algorithm and denote by $t'$ an arbitrary cell constructed from $t$ through a single iteration. In view of the possible types of partitions used in Interaction Forests, we see that the number of possible sets of the form $t' \cap G_n^{\#}$ is bounded by $2\cdot 7d(d-1)(g_n+1)^2$. By \Cref{cond:C2} we have $d=O(n^{K_0})$ and thus \Cref{cond:T_dim} is established. To show that \Cref{cond:T_boundary} is fulfilled, observe that the set difference between the boundary of $t$ and the boundary of $t'$ is contained in a set of the form \begin{align*}
     \{ x \in [0,1]^d : x_{j_1} = s_1 \} \cup \{ x \in [0,1]^d : x_{j_2} = s_2 \}
 \end{align*}
 with $j_1,j_2 \in \{1,\dots,d\}$ and $s_1,s_2 \in [0,1]$. This implies that a suitable choice for $\mathcal{H}$ in \Cref{cond:T_boundary} is given by the (grid points in $Q^d$ corresponding to the) sets $G(j_1,q_1) \cup G(j_2,q_2)$, see \Cref{def:grid} where $j_1,j_2 \in \{1,\dots, d\}$ and $q_1,q_2 \in Q$ are arbitrary. It is not difficult to see that $\mathcal H$ fulfills the desired conditions.
\end{proofname}

\section{Formal Definition of a Split Determining Sequence}\label{sec:split_determining_sequences}
Below, we provide a formal definition of split determining sequences as introduced in \Cref{subsec:treegrowingrules}.
\begin{definition}
        Let $W, k \in \N$ and for each $l=0,2,\dots, k-2$ let 
        \begin{align*}
            r_{w,l} \in \left(\{1,\dots,d\} \times [0,1] \right)^{2^{l}} \ (w=1,\dots,W)
        \end{align*}
        be given. Denote by $\mathcal{R}_k$ be the family $(r_{w,l})_{w=1,\dots,W, l=0,2,\dots k-2}$. We call such a family $\mathcal{R}_k$ a \quot{split determining sequence}.      
    \end{definition}
    \begin{definition}
        Let $\mathcal{R}_k$ be a split determining sequence. We say that a tree growing rule (of depth $k$) is associated with $\mathcal{R}_k = (r_{w,l})_{w=1,\dots,W; l=0,2,\dots,k -2}$ if at any second level $l$ of the tree, starting from the root $\rootcell$, for each of the $2^l$ cells nodes in this level, that is for each $a\in \{1,2\}^l$ and corresponding $t=\rootcell_a$, a set of possible splits is given by $\{ \bc_{l,a}^1 ,\dots , \bc_{l,a}^W\}$. The splits $\bc_{l,a}^w$ are of the form $(\gamma,U)$ with \[ U = t^{(\gamma)}_{\text{min}} + \mathcal U(t^{(\gamma)}_{\text{max}} - t^{(\gamma)}_{\text{min} }) \] where $(\gamma, \mathcal U)$ correspond to the entries in $r_{w,l} = (r_{w,l,a})_{a \in \{1,2\}^l} \in \left( \{1,\dots, d\}\times[0,1] \right)^{2^l}$.
    \end{definition}
Let for $W,k\in \N$, $w=1,\dots,W$, $l=0,2,\dots, k-2$ and $a \in \{1,2\}^l$
   \begin{align*}
        R_{w,l,a} \iid \text{Unif}( \{1,\dots,d\} \times [0,1])
    \end{align*}
Set $R^*_{w,l} := (R_{w,l,a})_{a\in\{1,2\}^l }$. Then, the family $\bR_k$ introduced in \Cref{subsec:treegrowingrules} is given by $\bR_k = (R^*_{w,l})_{w=1,\dots,W; l=0,2,\dots, k-2}$.

\vskip 0.2in
\bibliography{library}

\end{document}